\documentclass[
    a4paper,
    10.5pt,
    oneside,
    final		
]{memoir}

\usepackage[utf8]{luainputenc}

\usepackage[table]{xcolor}

\usepackage{XCharter}

\usepackage[tracking=true]{microtype}

\usepackage{pifont}

\usepackage[ngerman, UKenglish]{babel}

\usepackage{amsthm}
\usepackage{amsmath}
\newtheorem{theorem}{Theorem}[chapter]
\newtheorem{lemma}{Lemma}[chapter]
\newtheorem{corollary}{Corollary}[chapter]

\newtheorem{assumption}{Assumption}[chapter]

\newtheorem{remark}{Remark}[chapter]
\usepackage{amssymb}

\usepackage{algorithm}
\usepackage[noend]{algpseudocode}

\usepackage{graphicx}
\usepackage{caption}
\usepackage{subcaption}

\usepackage[autostyle]{csquotes}

\usepackage{siunitx}

\usepackage{booktabs}

\usepackage{tabulary}
\newcolumntype{C}{>{\centering\arraybackslash}X}
\newcolumntype{L}{>{\raggedright\arraybackslash}X}

\usepackage[shortlabels]{enumitem}

\usepackage[  
  pdflang={en-UK}, 
  bookmarksopen, 
  bookmarksnumbered, 
]{hyperref}

\definecolor{linkColor}{RGB}{0, 0, 9}

\usepackage[
    backend=biber, 
    style=numeric, 
    sortcites=true,  
    doi=false, 
    url=false, 
    urldate=iso, 
    seconds=true, 
    citereset=none, 
    maxbibnames=99, 
    giveninits=true, 
]{biblatex}
\AtEveryCitekey{\clearfield{eprint}} 
\newcommand{\citep}{\parencite}
\newcommand{\citet}{\textcite}

\hypersetup{
  pdftitle={Regularized Q-learning through Robust Averaging}, 
  pdfsubject={Research Paper},
  pdfauthor={Peter Schmitt-F\"orster and Tobias Sutter},
  pdfkeywords={
    Reinforcement Learning, 
    Q-Learning, 
    Estimation Bias, 
    Regularization, 
    Distributionally Robust Optimization, 
},
  pdfdisplaydoctitle=true, 
  bookmarksnumbered, 
  pdfstartview={FitH}, 
  colorlinks, 
  citecolor=black, 
  filecolor=black, 
  linkcolor=black, 
  urlcolor=linkColor, 
  breaklinks=true, 
  hypertexnames=false 
}

\usepackage[
    capitalise,  
    noabbrev,  
]{cleveref}

\usepackage{bookmark}

\renewenvironment{abstract}
{\centerline{\large\textbf{Abstract}}\vspace{0.7ex}%
  \bgroup\leftskip 20pt\rightskip 20pt\small\noindent\ignorespaces}%
{\par\egroup\vskip 0.25ex}

\newenvironment{keywords}
{\bgroup\leftskip 20pt\rightskip 20pt \small\noindent{\textbf{Keywords:}} }%
{\par\egroup\vskip 0.25ex}

\newcommand{\mc}{\mathcal}
\newcommand{\mb}{\mathbb}
\newcommand{\tr}[1]{{\textnormal{tr}}\left(#1\right)}

\newcommand\blfootnote[1]{%
  \begingroup
  \renewcommand\thefootnote{}\footnote{#1}%
  \addtocounter{footnote}{-1}%
  \endgroup
}

\setulmarginsandblock{2.5cm}{2.5cm}{*}
\setlrmarginsandblock{2.5cm}{2.5cm}{*}
\checkandfixthelayout
\chapterstyle{article}
\setsecheadstyle{\large\bfseries\raggedright}  

\let\originalsection\section
\let\originalchapter\chapter
\renewcommand{\subsection}{\originalsection} 
\renewcommand{\section}{\originalchapter}

\author{Peter Schmitt-F\"orster \and Tobias Sutter}
\date{
    \small{
        Department of Computer Science, University of Konstanz\\
        \{peter.schmitt-foerster,tobias.sutter\}@uni-konstanz.de\\
    }
}
  
\title{\textbf{Regularized Q-learning through Robust Averaging}}

\begin{document}

\maketitle
\blfootnote{This work was supported by the DFG in the Cluster of Excellence EXC 2117 “Centre for the Advanced Study of Collective Behaviour” (Project-ID 390829875).}

\begin{abstract}
    We propose a new Q-learning variant, called \textit{2RA Q-learning}, that addresses some weaknesses of existing Q-learning methods in a principled manner. 
    One such weakness is an underlying estimation bias which cannot be controlled and often results in poor performance. We propose a distributionally robust estimator for the maximum expected value term, which allows us to precisely control the level of estimation bias introduced. 
    The distributionally robust estimator admits a closed-form solution such that the proposed algorithm has a computational cost per iteration comparable to Watkins' Q-learning. 
    For the tabular case, we show that 2RA Q-learning converges to the optimal policy and analyze its asymptotic mean-squared error.
    Lastly, we conduct numerical experiments for various settings, which corroborate our theoretical findings and indicate that 2RA Q-learning often performs better than existing methods.
\end{abstract}

\section{Introduction}\label{sec:intro}

The optimal policy of a Markov Decision Process (MDP) is characterized via the dynamic programming equations introduced by \citet{Bellman:DynamicProgramming-57}.
While these dynamic programming equations critically depend on the underlying model, model-free reinforcement learning (RL) aims to learn these equations by interacting with the environment without any knowledge of the underlying model \citep{ref:BerTsi-96,ref:Sutton-1998,ref:book:meyn-22}.
There are two fundamentally different notions of interacting with the unknown environment. The first one is referred to as the synchronous setting, which assumes sample access to a generative model (or simulator), where the estimates are updated simultaneously across all state-action pairs in every iteration step. 
The second concept concerns an asynchronous setting, where one only has access to a single trajectory of data generated under some fixed policy. 
A learning algorithm then updates its estimate of a single state-action pair using one state transition from the sampled trajectory in every step. 
In this paper, we focus on the asynchronous setting, which is the considerably more challenging task than
the synchronous setting due to the Markovian nature of its sampling process.

One of the most popular RL algorithms is Q-learning \citep{ref:Watkins-PhD-89, ref:Watkins-92}, which iteratively learns the value function and hence the corresponding optimal policy of an MDP with unknown transition kernel. 
When designing an RL algorithm, there are various desirable properties such an algorithm should have, including (i) convergence to an optimal policy, (ii) efficient computation, and (iii) ``good" performance of the learned policy after finitely many iterations. 
Watkins' Q-learning is known to converge to the optimal policy under relatively mild conditions \citep{ref:Tsitsiklis-94} and a finite-time analysis is also available \citep{Eyal-01,ref:Srikant-12,ref:Qu-20}. 
Moreover, its simple update rule requires only one single maximization over the action space per step. 
The simplicity of Watkins' Q-learning, however, comes at the cost of introducing an overestimation bias \citep{ref:Thrun-93,ref:Hasselt-10}, which can severely impede the quality of the learned policy \citep{ref:Szita-08,ref:strehl-09,ref:Thrun-93,ref:Hasselt-10}.
It has been experimentally demonstrated that both, overestimation and underestimation bias, may improve learning performance, depending on the environment at hand (see \citet[Chapter~6.7]{ref:Sutton-1998} and \citet{ref:White-20} for a detailed explanation). Therefore, deriving a Q-learning method equipped with the possibility to precisely control the (over- and under-)estimation bias is desirable.

\paragraph{Related Work.} 

In the last decade, several Q-learning variants have been proposed to improve the weakness of Watkins' Q-learning while aiming to admit the desirable properties (i)-(iii). 
We discuss approaches that are most relevant to our work.
Double Q-learning \citep{ref:Hasselt-10} mitigates the overestimation bias of Watkins' Q-learning by introducing a double estimator for the maximum expected value term. 
While Double Q-learning is known to converge to the optimal policy and has a similar computational cost per iteration to Q-learning, it, unfortunately, introduces an underestimation bias which, depending on the environment considered, can be equally undesirable as the overestimation bias of Watkins' Q-learning \citep{ref:White-20}. 

Maxmin Q-learning \citep{ref:White-20} works with $N$ state-action estimates, where $N$ denotes a parameter, and chooses the smallest estimate to select the maximum action. 
Maxmin Q-learning allows to control the estimation bias via the parameter $N$. 
However, it generally requires a large value of $N$ to remove the overestimation bias. 
Conceptually, Maxmin Q-learning is related to our approach, where we select a maximum action based on the average current Q-function and then consider a worst-case ball around this average value. 
Similarly, REDQ \citep{redQ-21} updates $N$ Q-functions based on a max-action, min-function step over a sampled subset of size $M$ out of the $N$ Q-functions and allows to control the over- and underestimation bias. 
It further incorporates multiple randomized update steps in each iteration which results in good sample efficiency.
REDQ performs well in complicated non-tabular settings (including continuous state/action spaces). 
In the tabular setting, Maxmin Q-learning and REDQ are equipped with an asymptotic convergence proof.
Averaged-DQN \citep{ref:average:Q-17} is a simple extension to the Deep Q-learning (DQN) algorithm \citep{mnih2015humanlevel}, based on averaging previously learned Q-value estimates, which leads to a more stable training procedure and typically results in an improved performance by reducing the approximation error variance in the target values. 
Averaged-DQN and other regularized variants of DQN, such as Munchausen reinforcement learning \citep{ref:Vieillard-20}, use a deep learning architecture and are not equipped with any theoretical guarantees about convergence or quality of the learned policy \citep{ref:Prashant-20}.

In general, averaging in Q-learning is a well-known variance reduction method.
A specific form of variance-reduced Q-learning is presented in \citet{ref:wainwright2019variance}, where it is shown that the presented algorithm has minimax optimal sample complexity.
Regularized Q-learning \citep{ref:reg:q:learning-22} studies a modified Q-learning algorithm that converges when linear function approximation is used.
It is shown that simply adding an appropriate regularization term ensures convergence of the algorithm in settings where the vanilla variant does not converge due to the linear function approximation used.
A slightly different objective, when modifying Q-learning schemes, is to robustify them against environment shifts, i.e., settings where the environment, in which the policy is trained, is different from the environment in which the policy will be deployed. 
A popular approach is to consider a distributionally robust Q-learning model, where the resulting Q-function converges to the optimal Q-value that is robust against small shifts in the environment, see \citet{ref:Liu-22} for KL-based ambiguity sets and for distributionally robust formulations using the Wasserstein distance \citet{ref:Neufeld-22}.
The recent paper from \citet{ref:Liu-22} presents a distributionally robust Q-learning methodology, where the resulting Q-function converges to the optimal Q-value that is robust against small shifts in the environment. 
\citet{ref:Prashant-20}, \citet{ref:book:meyn-22}, \citet{ref:Lu-22}, and \citet{ref:lu2023convex} have introduced a new class of Q-learning algorithms called convex Q-learning which exploit a convex reformulation of the Bellman equation via the well-known linear programming approach to dynamic programming~\citep{hernandez2012discrete,ref:Hernandez-99}.

\paragraph{Contribution.}

In this paper, we introduce a new Q-learning variant called \textit{Regularized Q-learning through Robust Averaging} (2RA), which combines regularization and averaging.
The proposed method has two parameters, $\rho>0$ quantifying the level of robustness/regularization introduced and $N\in\mathbb{N}$, which describes the number of state-action estimates used to form the empirical average. 
Centered around this new Q-learning variant, our main contributions can be summarized as follows:
\begin{itemize}[leftmargin=*]
    \itemsep0.1em 
    \item We present a tractable formulation of the proposed 2RA Q-learning where the computational cost per iteration is comparable to Watkins' Q-learning.
    \item For any choice of $N$ and for any positive sequence of regularization parameters $\{\rho_n\}_{n\in\mathbb{N}}$ such that $\lim_{n\to\infty}\rho_n=0$, we prove that the proposed 2RA Q-learning asymptotically converges to the true $Q$-function, see Theorem~\ref{thm:convergence}. 
    \item We show how the choice of the two parameters $\rho$ and $N$ allow us to control the level of estimation bias in 2RA Q-learning, see Theorem~\ref{thm:estimation:bias}, and show that as $N\to\infty$ our proposed estimation scheme becomes unbiased.
    \item We prove that under certain technical assumptions, the asymptotic mean-squared error of 2RA Q-learning is equal to the asymptotic mean-squared error of Watkins' Q-learning, provided that we choose the learning rate of our method $N$-times larger than that of Watkins' Q-learning, see Theorem~\ref{thm:AMSE}. 
    This theoretical insight allows practitioners to start with an initial guess of the learning rate for the proposed method that is $N$-times larger than that of standard Q-learning.
    \item We demonstrate that the theoretical properties of 2RA can be numerically reproduced in synthetic MDP settings. 
    In more practical experiments from the OpenAI gym suite \citep{gym2016} we show that, even when implementations require deviations from out theoretically required assumptions, 2RA Q-learning has good performance and mostly outperforms other Q-learning variants.
\end{itemize}

\section{Problem Setting}\label{sec:setting}

Consider an MDP given by a six-tuple $\left(\mathcal{S}, \mathcal{A}, P, r, \gamma, s_0\right)$ comprising a finite state space $\mathcal{S}=\{1,\ldots,S\}$, a finite action space $\mathcal{A}=\{1,\ldots,A\}$, a transition kernel $P: \mathcal{S} \times \mathcal{A} \rightarrow \Delta(\mathcal{S})$, a reward-per-stage function $r: \mathcal{S} \times \mathcal{A} \rightarrow \mathbb{R}$,
a discount factor $\gamma\in(0,1)$, and a deterministic initial state $s_0\in\mathcal{S}$. 
Note that $\left(\mathcal{S}, \mathcal{A}, P, r, \gamma,s_0\right)$ describes a controlled discrete-time stochastic system, where the state and the action applied at time $t$ are denoted as random variables $S_t$ and $A_t$, respectively.
If the system is in state $s_t\in\mc S$ at time $t$ and action $a_t\in\mc A$ is applied, then an immediate reward of $r(s_t,a_t)$ is incurred, and the system moves to state $s_{t+1}$ at time $t+1$ with probability $P(s_{t+1}|s_t,a_t)$. 
Thus, $P(\cdot|s_t,a_t)$ represents the distribution of $S_{t+1}$ conditional on $S_t=s_t$ and $A_t=a_t$. 
It is often convenient to represent the transition kernel as a matrix $P\in\mathbb{R}^{SA\times S}$.
Actions are usually chosen according to a policy that prescribes a random action at time $t$ depending on the state history up to time $t$ and the action history up to time $t-1$.
Throughout the paper, we restrict attention to stationary Markov policies, which are described by a stochastic kernel $\pi:\mathcal{S}\to\Delta(\mc A)$, that is, $\pi(a_{t}|s_{t})$ denotes the probability of choosing action $a_{t}$ while being in state~$s_{t}$.
We denote by $\Pi$ the space of all stationary Markov policies. 
Given a stationary policy $\pi$ and an initial condition $s_0$, it is well-known that there exists a unique probability measure $\mathbb{P}_{s_0}^\pi$ defined on the canonical sample space $\Omega=(\mathcal{S}\times\mathcal{A})^\infty$ equipped with its power set $\sigma$-field $\mc F=2^{\Omega}$, such that for all $t\in\mb N$ we have (see \citet[Section 2.2]{hernandez2012discrete} for further details)
\begin{align*}
    &\mathbb{P}_{s_0}^\pi(S_0=s_0)=1,\\
    &\mathbb{P}_{s_0}^\pi(S_{t+1}=s_{t+1}|S_t = s_t, A_t=a_t) = P(s_{t+1}|s_t,a_t) \qquad\forall s_t,s_{t+1}\in\mc S,\;a_t\in\mc A, \\
    &\mathbb{P}_{s_0}^\pi(A_{t}=a_{t}|S_{t}=s_{t})= \pi(a_{t}|s_{t})\quad\forall s_{t}\in\mc S,\;a_t\in\mc A.
\end{align*} 
To keep the notation simple, in the following, we denote $\mathbb{P}_{s_0}^\pi$ by $\mathbb{P}$ and the corresponding expectation operator by $\mathbb{E}$.
Then, the ultimate goal is to find an optimal policy $\pi^\star$ which leads to the largest expected infinite-horizon reward, i.e., 
\begin{equation}\label{eq:optimal:policy}
    \pi^\star \in \arg\max_{\pi\in\Pi} \sum_{t=0}^{\infty} \gamma^t \mathbb{E}[r(S_t,A_t)].
\end{equation}
An optimal (deterministic) policy can be alternatively obtained from the optimal Q-function as $\pi^\star(\cdot |s) = \delta_{a^\star}(\cdot)$, where $a^\star \in \arg\max_{a\in\mc A} Q^\star(s,a)$ and the optimal Q-function satisfies the Bellman equation \citep{ref:BerTsi-96}, i.e., $\forall (s, a)\in \mc S \times \mc A$
\begin{equation}\label{eq:Bellman:eq}
    Q^\star(s,a) = r(s,a) + \gamma \sum_{s'\in\mc S} \! P(s'|s,a) \max_{a'\in\mc A} Q^\star (s',a').
\end{equation}
Solving for the Q-function via \eqref{eq:Bellman:eq} requires the knowledge of the underlying transition kernel $P$ and reward function $r$, objects which in reinforcement learning problems generally are not known.

In this work, we focus on the so-called \textit{asynchronous} RL setting, where the Q-function is learned from a single trajectory of data which we assume to be generated from a fixed behavioral policy leading to state-action pairs $\{(S_1,A_1),\dots,(S_n,A_n),\dots\}$.  
The standard asynchronous Q-learning algorithm \citep{ref:Wentao-20} can then be expressed as
\begin{equation}\label{eq:Q:learning:watkins}
    \begin{aligned}
        &Q_{n+1}(S_n, A_n) = Q_{n}(S_n, A_n) + \alpha_n^{\mathsf{QL}} \big( r(S_n,A_n) + \gamma \max_{a'\in\mc A} Q_{n}(S_{n+1}, a')-Q_{n}(S_n, A_n) \big),
    \end{aligned}
\end{equation}
where $\alpha_n^{\mathsf{QL}} \in(0,1]$ is the learning rate. 
It has been shown \citep{ref:Tsitsiklis-94,ref:Csaba-97,ref:Qu-20,ref:Lee-20} that if each state is updated infinitely often and each action is tried an infinite number of times in each state, convergence to the optimal Q-function can be obtained. 
That is, for a learning rate satisfying $\sum_{n=0}^\infty \alpha_n^{\mathsf{QL}} = \infty$ and $\sum_{n=0}^\infty (\alpha_n^{\mathsf{QL}})^2 <\infty$, Q-learning converges $\mathbb{P}$-almost surely to an optimal solution of the Bellman equation \eqref{eq:Bellman:eq}, i.e., $\lim_{n\to\infty} Q_n = Q^\star$ $\mathbb{P}$-almost surely.
To simplify notation, we introduce the state-action variables $X_n = (S_n,A_n)$ and define $\mc X = \mc S \times \mc A$. As the state space in practice is often large, the Q-functions are commonly approximated via fewer basis functions. 
When interpreting the Q-function as a vector on $\mathbb{R}^{|\mc X|}$, we use 
\begin{equation}\label{eq:q:approx}
    Q^\star \approx \Phi^\top \theta^\star,  \quad \theta^\star\in\mathbb{R}^d, \ \Phi = \left(\phi(s_1,a_1),\dots,\phi(s_{|\mc S|},a_{|\mc A|})\right)\in\mathbb{R}^{d\times |\mc X|},
\end{equation}
where with slight abuse of notation we denote by $\phi(s,a)\in\mathbb{R}^d$ the given feature vectors associated with the pairs $s=s_i$ and $a=a_i$ for $i\in\{1,\dots |\mc X|\}$. 
Clearly, by choosing $d=|\mc X|$ and the canonical feature vectors $\phi(s,a) = \sum_{i=1}^{|\mc X|} e_i \cdot \mathsf{1}_{\{ (s,a) = (s_i,a_i) \}}$ for $i=1,\dots,|\mc X|$, the approximation \eqref{eq:q:approx} is exact, which is referred to as the tabular setting. For our convergence results that only hold in the tabular setting, we will use this representation.
In this linear function approximation formulation, the standard asynchronous Q-learning \eqref{eq:Q:learning:watkins} can be expressed as so-called $Q(0)$-learning \citep{ref:Melo-08, ref:book:meyn-22}, which is given as 
\begin{equation}\label{eq:q-learning}
    \theta_{n+1} = \theta_n + \alpha_n^{\mathsf{QL}}\big( b(X_n) - A_1(X_n)\theta_n + \mc E(X_n,S_{n+1},\theta_n)\big),
\end{equation}
where $b(X_n) = \phi(X_n) r(X_n)$, $A_1(X_n) = \phi(X_n)\phi(X_n)^\top$, $\mc E(X_n,S_{n+1},\theta_n) = \gamma \phi(X_n)\max_{a'\in\mc A} \phi(S_{n+1},a')^\top\theta_n$.
It is well-known \citep{ref:Hasselt-10}, that the term $\mc E$ in the Q-learning \eqref{eq:q-learning} introduces an overestimation bias when compared to \eqref{eq:Bellman:eq}, since
\begin{equation}\label{eq:Jensen:overestimation}
    \max_{a'\in\mc A} \mathbb{E}[\phi(s,a')^\top\theta_n] \leq \mathbb{E}[\max_{a'\in\mc A} \phi(s,a')^\top\theta_n] \ \ \forall s\in\mc S,
\end{equation}
where the expectation is with respect to the random variable $\theta_n$ defined according to \eqref{eq:q-learning} and the inequality is due to Jensen. 
A common method to mitigate this overestimation bias is to modify Q-learning to the so-called Double Q-learning \citep{ref:Hasselt-10}, which using the linear function approximation, can be expressed in the form \eqref{eq:q-learning}, where
\begin{align*}
    &\theta_n = \begin{pmatrix}
    \theta_n^A \\ \theta_n^B
    \end{pmatrix}, 
    \quad    \mc E(X_n,S_{n+1},\theta_n) = \begin{pmatrix}
    \beta_n \gamma \phi(X_n)\phi(S_{n+1},\pi_{\theta_n^A}(S_{n+1}))^\top \theta_n^B \\
    (1-\beta_n) \gamma \phi(X_n)\phi(S_{n+1},\pi_{\theta_n^B}(S_{n+1}))^\top \theta_n^A
    \end{pmatrix}, \\
    &b(X_n) = \begin{pmatrix}
    \beta_n \phi(X_n) r(X_n) \\
    (1-\beta_n) \phi(X_n) r(X_n)
    \end{pmatrix},
    \quad A_1(X_n)  = \begin{pmatrix}
    \beta_n \phi(X_n) \phi(X_n)^\top \! & 0 \\
    0 & \!\!(1-\beta_n) \phi(X_n) \phi(X_n)^\top\!
    \end{pmatrix},
\end{align*}
and $\beta_n$ are i.i.d.~Bernoulli random variables with equal probability, and $\pi_\theta(s) = \arg\max_{a\in\mc A}\phi(s,a)^\top \theta$. 
 While Double Q-learning avoids an overestimation bias, it introduces an underestimation bias  \citep[Lemma~1]{ref:Hasselt-10}, as each component of $\mc E$ satisfies\footnote{Analogously we have $\mathbb{E}[ \phi(s,\pi_{\theta^B_n}(s))^\top\theta^A_n] \leq \max_{a'\in\mc A} \mathbb{E}[\phi(s,a')^\top\theta^B_n]$ for all $s\in\mc S$.}
\begin{equation*}
    \mathbb{E}[ \phi(s,\pi_{\theta^A_n}(s))^\top\theta^B_n] \leq \max_{a'\in\mc A} \mathbb{E}[\phi(s,a')^\top\theta^A_n] \quad \forall  s\in\mc S.
\end{equation*}
It can be directly seen by the Jensen inequality~\eqref{eq:Jensen:overestimation} that in the special case of $\theta^A = \theta^B$ the inequality above becomes an equality.
An inherent difficulty with the overestimation bias of standard Q-learning (resp.~the underestimation bias of Double Q-learning) is that it cannot be controlled, i.e., the level of under-(resp.~over-)estimation bias depending on the problem considered can be significant. 
In the following, we present a Q-learning method where the level of estimation bias can be precisely adjusted via a hyperparameter.

\section{Regularization through Robust Averaging}\label{sec:robust:Q:learning}

We propose the 2RA Q-learning method defined by the update rule
\begin{equation}\label{eq:q-learning:2RA}
        \theta^{(i)}_{n+1} = \theta^{(i)}_n + \alpha_n \beta_n^{(i)}( b(X_n) - A_1(X_n)\theta^{(i)}_n + \mc E_{\rho}(X_n,S_{n+1},\widehat\theta_{N,n})), \quad \text{for }i=1,\dots, N,
\end{equation}
where $\beta_n$ is a generalized i.i.d.~Bernoulli random variable on $\{1,\dots,N\}$ with equal probability for each component $i$, i.e., $\mathbb{P}(\beta_n = e_i) = 1/N$ for all $i=1,\dots,N$, where $e_i$ is the i$^{\text{th}}$ unit vector on $\mathbb{R}^N$ and $\alpha_n$ is the learning rate. 
2RA Q-learning~\eqref{eq:q-learning:2RA} is based on the estimator defined for all $x\in\mc X, s'\in\mc S, \theta\in\mathbb{R}^d$ as
\begin{equation}\label{eq:robust:estimator}
    \mc E_{\rho}(x,s',\theta) = \gamma \phi(x) \max_{a'\in\mc A} \min_{\theta'\in\mc B_\rho(\theta)}\phi(s',a')^\top\theta', 
\end{equation}
where $\rho\geq0$ is a given parameter and the ambiguity (or uncertainty) set $\mc B_\rho(\theta)$ is assumed to be of the form $\mc B_\rho(\theta) =\{\theta'\in\mathbb{R}^d: \|\theta - \theta'\|_2^2 \leq \rho\}$ with its center $\theta$ being the empirical average 
\begin{equation}\label{eq:estimator:proposed:DRO}
    \widehat\theta_{N,n} = \frac{1}{N}\sum_{i=1}^N \theta_n^{(i)}.
\end{equation}
The intuition behind the proposed 2RA Q-learning \eqref{eq:q-learning:2RA} is to mitigate the overestimation bias of Watkins' Q-learning by approximating the term $\max_{a'\in\mathcal{A}}\mathbb{E}_{\mathbb{P}}[\phi(s,a')^\top\theta_n]$, where we consider $\theta_n$ as a $\mathbb{R}^d$-valued random variable distributed according to $\mathbb{P}$, via the distributionally robust model
\begin{equation}\label{eq:DRO:motivation}
    \max_{a'\in\mathcal{A}} \min_{\mathbb{Q}\in\mathbb{B}_\rho(\widehat{\mathbb{P}}_{N,n})}\mathbb{E}_{\mathbb{Q}}[\phi(s,a')^\top\theta_n],
\end{equation}
where $\mathbb{B}_\rho(\widehat{\mathbb{P}}_{N,n})$ is a set of probability measures centered around the empirical distribution $\widehat{\mathbb{P}}_{N,n} = \frac{1}{N}\sum_{i=1}^N \delta_{\theta^{(i)}_n}$.
When considering the ambiguity set $\mathbb{B}_\rho(\widehat{\mathbb{P}}_{N,n})$ as the ball of all distributions that have a fixed diagonal covariance and a 2-Wasserstein distance to $\widehat{\mathbb{P}}_{N,n}$ of at most $\sqrt{\rho}$, then by \citep[Theorem~2]{ref:Nguyen-risk-21} the distributionally robust model~\eqref{eq:DRO:motivation} directly corresponds to our estimator~\eqref{eq:robust:estimator}.  
Running the 2RA Q-learning \eqref{eq:q-learning:2RA} requires an evaluation of the estimator $\mc E_{\rho}(X_n,S_{n+1},\widehat\theta_{N,n})$ given by the optimization problem \eqref{eq:robust:estimator}, which admits a closed form expression.
\begin{lemma}[Estimator computation]\label{lem:computation}
The estimator defined in \eqref{eq:robust:estimator} is equivalently expressed as
    \begin{equation*}
        \mc E_{\rho}(x,s',\theta) = \gamma \phi(x) \max_{a'\in\mc A} \left\{ \phi(s',a')^\top \theta  - \sqrt{\rho} \|\phi(s',a')\|_2 \right\}.
    \end{equation*}
\end{lemma}
\begin{proof}
    According to \citet[Lemma~9.2]{bertsimas-LPbook}, for any $c,\theta'\in\mathbb{R}^d$ and positive definite matrix $H\in\mathbb{R}^{d\times d}$ the optimization problem
    \begin{equation*}
        \min_{\theta\in\mathbb{R}^d} \{c^\top \theta \ : \ (\theta'-\theta)^\top H^{-1} (\theta'-\theta) \leq \rho\}
    \end{equation*}
    admits a closed-form solution 
    \begin{equation*}
        \theta^\star = \theta' - \sqrt{\frac{\rho}{c^\top H c}}Hc.
    \end{equation*}
    Therefore, by setting $c = \phi(s',a')$ and $H$ to be the identity matrix, an optimizer in \eqref{eq:robust:estimator} is
    \begin{equation*}
        \theta^\star =  \theta' - \frac{\sqrt{\rho}}{\|\phi(s',a')\|_2}\phi(s',a'),
    \end{equation*}
    which completes the proof.
\end{proof}
In the tabular setting 2RA Q-learning~\eqref{eq:q-learning:2RA} even for $N=2$ is different from Double Q-learning \citep{ref:Hasselt-10}. 
However, in the special case where $\rho=0$ and $N=1$, our method collapses to Watkins' Q-learning~\eqref{eq:q-learning}.
In our proposed 2RA Q-learning, we've made a modification to the term $\mathcal{E}(X_n, S_{n+1}, \theta_n)$ from Wattkins Q-learning \eqref{eq:q-learning}. It is now replaced with a regularized version, denoted as $\mathcal{E}_\rho$ based on Lemma~\ref{lem:computation}. This adjustment is combined with the averaging property using $\widehat\theta_{N,n}$.
The regularization term $\sqrt{\rho} \|\phi(s',a')\|_2$ can be interpreted as negative UCB bonus term that discourages exploration, which has been considered for linear MDPs in \cite{jin2019provably}, see also \cite{ref:Qian-19}.

In the remainder of this section, we theoretically investigate 2RA Q-learning~\eqref{eq:q-learning:2RA} and, in particular, study how to choose the two regularization parameters $\rho$, $N$, and the learning rate $\alpha_n$. 
In Section~\ref{sec:asymptotic:convergence}, we show what properties the regularization term $\rho$ should satisfy such that 2RA Q-learning~\eqref{eq:q-learning:2RA} asymptotically converges to the optimal Q-function. 
This convergence is independent of the number $N$ of Q-function estimates. 
Section~\ref{ssec:estimation:bias} shows how the two terms $\rho$ and $N$ can be exploited to control the estimation bias of the presented scheme. 
Finally, Section~\ref{ssec:AMSE} studies the convergence rate via the notion of the asymptotic mean squared error and, in particular, shows how to choose the learning rate $\alpha_n$ as compared to Watkins' Q-learning.

\subsection{Asymptotic Convergence}\label{sec:asymptotic:convergence}

The 2RA Q-learning~\eqref{eq:q-learning:2RA} for the tabular setting actually converges almost surely to the optimal Q-function satisfying the Bellman equation~\eqref{eq:Bellman:eq}, provided the radius $\rho$ is chosen appropriately.
\begin{theorem}[Asymptotic convergence]\label{thm:convergence}
    Consider the tabular setting where $d=|\mc X|$, $\Phi$ is the canonical basis and let $\{\rho_n\}_{n\in\mathbb{N}}$ be a sequence of non-negative numbers such that $\lim_{n\to\infty}\rho_n = 0$. Moreover, assume that
    \begin{enumerate}[label=(\roman*)]
        \item The learning rates satisfy $\alpha_n(s,a)\in(0,1]$, $\sum_{n=0}^\infty \alpha_n(s,a)=\infty$, $\sum_{n=0}^\infty \alpha^2_n(s,a)<\infty$ and $\alpha_n(s,a) = 0$ unless $(s,a) = (S_n,A_n)$.
        \item The reward $r$ is bounded.
    \end{enumerate}
     Then, for any $N\in\mathbb{N}$, 2RA Q-learning~\eqref{eq:q-learning:2RA} converges to the optimal Q-function $Q^\star$, i.e., $\lim_{n\to\infty} \Phi^\top \widehat \theta_{N,n} = Q^\star$ $\mathbb{P}$-almost surely.
\end{theorem}

\noindent Note that $\Phi^\top \widehat \theta_{N,n}$ is our learned 2RA Q-function under the canonical basis describing the tabular setting.
The proof of Theorem~\ref{thm:convergence} is based on the following technical stochastic approximation result. 
We denote by $\|\cdot\|_w$ a weighted maximum norm with weight $w = (w_1,\dots,w_d)$, $w_i>0$. If $x\in\mathbb{R}^d$, then $\|x\|_w = \max_i \frac{|x_i|}{w_i}$.
\begin{lemma}[{\citet[Lemma~1]{ref:Csaba-00}}]\label{lemma:SA:result}
    Consider a stochastic process $\{(\alpha_n,\Delta_n,F_n)\}_{n\geq 0}$, where $\alpha_n,\Delta_n,F_n:\mathcal{X}\to\mathbb{R}$ satisfy the equations
    \begin{equation}\label{eq:SA:eq}
        \Delta_{n+1}(x) = (1-\alpha_n(x))\Delta_n(x) + \alpha_n(x)F_n(x), \quad x\in \mathcal{X}, \ n=0,1,\dots 
    \end{equation}
    Let $P_n$ be a sequence of increasing $\sigma$-fields such that $\alpha_0$ and $\Delta_0$ are $P_0$-measurable and $\alpha_n,\Delta_n$ and $F_{n-1}$ are $P_n$-measurable for $n=1,2,\dots$. Assume the following hold
    \begin{enumerate}[label=(\roman*)]
        \item \label{item:1:SA} the set $\mathcal{X}$ is finite;
        \item \label{item:2:SA} $\alpha_n(x)\in(0,1] $, $\sum_{n=1}^\infty \alpha_n(x) = \infty$, $\sum_{n=1}^\infty \alpha^2_n(x) < \infty$ almost surely;
        \item \label{item:3:SA} $\|\mathbb{E}[F_n(\cdot)|P_n]\|_w\leq \kappa \|\Delta_n\|_w + c_n$, where $\kappa\in[0,1)$ and $c_n$ converges to zero almost surely;
        \item \label{item:4:SA} $\mathsf{Var}(F_n(x)|P_n)\leq K(1 + \|\Delta_n\|_w)^2$, where $K$ is some constant.
    \end{enumerate}
    Then, $\Delta_n$ converges to zero almost surely as $n\to\infty$.
\end{lemma}
\begin{proof}[Proof of Theorem~\ref{thm:convergence}]\ 
    The proof builds up on the convergence results of SARSA \citep{ref:Csaba-00}, Double Q-learning \citep{ref:Hasselt-10} and uses Lemma~\ref{lemma:SA:result} as a key ingredient. For the convenience of notation, we carry out the proof in the Q-function notation. 
    That is, in the tabular setting, where no function approximation is applied, by invoking Lemma~\ref{lem:computation}, the proposed 2RA Q-learning~\eqref{eq:q-learning:2RA} is expressed as
    \begin{equation}\label{eq:proposed:Q}
        \begin{aligned}
            Q^{(i)}_{n+1}(S_n, A_n)
            &= Q^{(i)}_{n}(S_n, A_n)  \\
            &\quad + \alpha_n \beta_n^{(i)} \big( r(S_n,A_n)
            + \gamma( \max_{a'\in\mc A} \widehat Q_{N,n}(S_{n+1}, a') 
             -\sqrt{\rho_n})-Q^{(i)}_{n}(S_n, A_n) \big),
        \end{aligned}
    \end{equation}
    where $\widehat Q_{N,n}(s, a) = \frac{1}{N}\sum_{i=1}^N Q_n^{(i)}(s,a)$ for $s\in\mc S$, $a\in\mc A$. 
    In the following, we fix an arbitrary index $i\in\{1,\dots,N\}$ and with regard to Lemma~\ref{lemma:SA:result}, we define $P_n$ as the $\sigma$-field generated by $\{S_n, A_n, \alpha_n, \dots, S_0, A_0,$ $\alpha_0, Q_0^{(1)}, \dots,Q_0^{(N)}\}$, $\mathcal{X} = \mc S\times \mc A$, $\Delta_n = Q_n^{(i)} - Q^\star$ and $F_n(S_n,A_n) = r(S_n,A_n) + \gamma (\max_{a'\in\mc A} \widehat Q_{N,n}(S_{n+1},a')-\sqrt{\rho_n}) - Q^\star(S_n,A_n)$. 
    Then, 2RA Q-learning~\eqref{eq:q-learning:2RA} can be expressed as an instance of \eqref{eq:SA:eq}. 
    To apply Lemma~\ref{lemma:SA:result}, we need to ensure its assumptions are satisfied. Assumption \ref{item:1:SA} and \ref{item:2:SA} clearly hold.
    
    To show Assumption~\ref{item:3:SA}, we note that the term $F_n$ can be alternatively expressed as
    \begin{equation}\label{eq:Fn:decomp}
       F_n(S_n,A_n) = F_n^{Q^{(i)}}(S_n,A_n) + \gamma\left(\max_{a'\in\mc A}\widehat Q_{N,n}(S_{n+1},a') - \max_{a'\in\mc A} Q_n^{(i)}(S_{n+1},a') -\sqrt{\rho_n} \right),
    \end{equation}
    with 
    \begin{equation}
        F_n^{Q^{(i)}}(S_n,A_n) = r(S_n, A_n) + \gamma \max_{a'\in\mc A} Q_n^{(i)}(S_{n+1},a') - Q^\star(S_n,A_n),
    \end{equation}
    where the term $F_n^{Q^{(i)}}(S_n,A_n)$ corresponds to Watkins' Q-learning for $Q_n^{(i)}$. 
    Therefore, it is well-known \citep[Theorem~2]{ref:Jaakkola-94} that 
    $\|\mathbb{E}[F_n^{Q^{(i)}}(\cdot)|P_n]\|_w\leq\gamma \|\Delta_n\|_w$, which via \eqref{eq:Fn:decomp} implies that $\|\mathbb{E}[F_n(\cdot)|P_n]\|_w\leq\gamma \|\Delta_n\|_w + c_n$, where 
    \begin{equation}\label{eq:sequence:cn}
        c_n = \gamma\mathbb{E}\left[\max_{a'\in\mc A}\widehat Q_{N,n}(S_{n+1},a') - \max_{a'\in\mc A} Q_n^{(i)}(S_{n+1},a') -\sqrt{\rho_n} \ | \ P_n \right]. 
    \end{equation}
    It remains to show that $c_n$ converges to zero almost surely. 
    Recall that by assumption $\lim_{n\to\infty}\rho_n=0$.
    We define
    \begin{equation*}
        \delta_n^{(i)}(s,a) = Q_{n}^{(i)}(s,a) - \widehat Q_{N,n}(s,a)
    \end{equation*}
    and will show that $\lim_{n\to\infty}\|\delta_n^{(i)}(\cdot,\cdot)\|=0$ almost surely. 
    The reverse triangle inequality applied to the $\infty$-norm implies that $\mathbb{P}$-almost surely
    \begin{equation*}
        \lim_{n\to\infty} \left( \max_{a'\in\mc A} Q_n^{(i)}(S_{n+1},a') - \max_{a'\in\mc A}\widehat Q_{N,n}(S_{n+1},a') \right) = 0
    \end{equation*}
    and hence that $c_n$ converges to zero almost surely.
    
    We distinguish two cases. First, we consider an update on component $i$, then by \eqref{eq:proposed:Q}
    \begin{align*}
       \delta_{n+1}^{(i)}(s,a) &=Q_{n+1}^{(i)}(s,a) - \widehat Q_{N,n+1}(s,a) \\
       &= Q_{n}^{(i)}(s,a) + \alpha_n \left(r(s,a) + \gamma \max_{a'\in\mc A} \widehat Q_{N,n}(s',a')- \gamma \sqrt{\rho_n} - Q_{n}^{(i)}(s,a) \right)  - \widehat Q_{N,n}(s,a)\\ 
       &\qquad - \frac{1}{N}\left(Q_{n+1}^{(i)}(s,a) - Q_{n}^{(i)}(s,a)\right) \\
       &= \delta_{n}^{(i)}(s,a) + \left(\alpha_n - \frac{\alpha_n}{N}\right)\left(r(s,a) + \gamma \max_{a'\in\mc A} \widehat{Q}_{N,n}(s',a')- \gamma \sqrt{\rho_n} - Q_{n}^{(i)}(s,a)\right),
    \end{align*}
    where the second equality follows from the decomposition $\frac{1}{N}\sum_{j=1}^N Q^{(j)}_{n+1}(s,a) = \frac{1}{N}(\sum_{j\neq i} Q_n^{(j)}+ Q_{n+1}^{(i)})$. 
    The third equality then uses our proposed Q-learning update formula \eqref{eq:proposed:Q} given as $Q^{(i)}_{n+1}(s,a) = Q^{(i)}_{n}(s,a) + \alpha_n ( r(s,a) + \gamma \max_{a'\in\mathcal{A}} \widehat Q_{N,n}(s',a') - \gamma \sqrt{\rho_n}- Q_n^{(i)}(s,a) )$.
    On the other hand, if the update is performed on component $j\neq i$, then 
    \begin{align*}
       \delta_{n+1}^{(i)}(s,a) &=Q_{n+1}^{(i)}(s,a) - \widehat Q_{N,n+1}(s,a) \\
       &= Q_{n}^{(i)}(s,a)  - \widehat Q_{N,n}(s,a) - \frac{1}{N}(Q_{n+1}^{(j)}(s,a) - Q_{n}^{(j)}(s,a)) \\
       &= \delta_{n}^{(i)}(s,a) - \frac{\alpha_n}{N}(r(s,a) + \gamma \max_{a'\in\mc A} \widehat Q_{N,n}(s',a')- \gamma \sqrt{\rho_n} - Q_{n}^{(j)}(s,a) ).
    \end{align*}
    Hence, in total, we get
    \begin{align*}
        &\mathbb{E}[\delta_{n+1}^{(i)}(s,a)|P_n]\\ 
        &\quad = \frac{1}{N}\mathbb{E}\left[ \delta_{n}^{(i)}(s,a) + \frac{N-1}{N}\alpha_n(r(s,a) + \gamma \max_{a'\in\mc A} \widehat Q_{N,n}(s',a')- \gamma \sqrt{\rho_n} - Q_{n}^{(i)}(s,a))|P_n \right] \\
        &\quad \qquad + \frac{1}{N} \sum_{j\neq i}\mathbb{E}\left[\delta_{n}^{(i)}(s,a) - \frac{\alpha_n}{N}(r(s,a) + \gamma \max_{a'\in\mc A} \widehat Q_{N,n}(s',a')- \gamma \sqrt{\rho_n} - Q_{n}^{(j)}(s,a) )|P_n \right]\\
        &\quad =\mathbb{E}\left[(1-\frac{\alpha_n}{N})\delta_{n}^{(i)}(s,a)|P_n\right] \\
        &\quad =(1-\frac{\alpha_n}{N})\mathbb{E}\left[\delta_{n}^{(i)}(s,a)|P_{n-1}\right],
    \end{align*}
    where the second equality follows from the observation that $\sum_{j\neq i}Q_{n}^{(j)} = N \widehat Q_{N,n} - Q_{n}^{(i)}$. 
    Recall that $\delta^{(i)}_n(s,a) = \mathbb{E}[\delta_{n}^{(i)}(s,a)|P_{n-1}]$ for any $n$.
    Hence, we have derived the update rule
    \begin{equation}\label{eq:update:bar:delta}
        \delta^{(i)}_{n+1}(s,a) = (1-\frac{\alpha_n}{N})\delta^{(i)}_{n}(s,a),
    \end{equation}
    which directly implies that    
    $\lim_{n\to\infty}\delta_{n}^{(i)}(s,a) = 0$ almost surely for all $(s,a)\in\mathcal{S}\times\mathcal{A}$. 
    Since $\mathcal{S}$ and $\mathcal{A}$ are finite sets this implies that $\lim_{n\to\infty}\|\delta_n^{(i)}\|=0$ almost surely as desired. 
    Hence, $\lim_{n\to\infty} c_n = 0$ almost surely, which ensures Assumption~\ref{item:3:SA}.
    
    We finally show that Assumption~\ref{item:4:SA} holds. 
    Again we use the decomposition \eqref{eq:Fn:decomp}. Since the reward $r$ is assumed to be bounded, it is known \citep{ref:Csaba-00} that
    \begin{equation}\label{eq:var:vanilla:Q}
        \mathsf{Var}(F_n^{Q^{(i)}}(x)|P_n)\leq K(1+\|\Delta_n\|_w)^2,
    \end{equation}
    where again $\Delta_n = Q_n^{(i)} - Q^\star$. 
    We next show that there exists a constant $K_2$ such that
    \begin{equation}\label{eq:var:ass:iiii}
        \mathsf{Var}\left(\max_{a'\in\mc A}\widehat Q_{N,n}(S_{n+1},a') - \max_{a'\in\mc A} Q_n^{(i)}(S_{n+1},a')|P_n\right)\leq K_2(1+\|\Delta_n\|_w)^2.
    \end{equation}
    By the Cauchy-Schwarz inequality \eqref{eq:var:vanilla:Q} and \eqref{eq:var:ass:iiii} imply Assumption~\ref{item:4:SA}.
    To establish \eqref{eq:var:ass:iiii}, we show that the Q-functions $Q_n^{(i)}$ are bounded for any $n$ and any $i$.
    Such results are well-known for classical Q-learning, see \cite{ref:Gosavi-06}. For our modified Q-learning, we can show it via a simple contradiction argument.
    Suppose for the sake of contradiction there is an index $i$ such that $\lim_{n\to\infty}\|Q_n^{(i)}\|=\infty$. 
    We first consider the setting, where there is an $s'\in\mc S$ and $a'\in\mc A$ such that $\lim_{n\to\infty}Q_n^{(i)}(s',a') = \infty$. 
    So there exists $\bar n\in\mathbb{N}$ such that for all $n\geq \bar n$ we have $Q_n^{(j)}(s',a) \leq Q_n^{(i)}(s',a')$ for all $a\in\mathcal{A}$ and for all $j=1,\dots,N$. 
    Therefore, $\max_{\bar{a}\in\mathcal{A}}\widehat Q_{N,n}(s',\bar{a}) \leq 
    Q_n^{(i)}(s',a') \leq \max_{\bar{a}\in\mathcal{A}} Q_n^{(i)}(s',\bar{a})$, which implies
    \begin{subequations}
        \begin{align}
           Q_{n+1}^{(i)}(s,a) 
           & =Q_{n}^{(i)}(s,a) + \alpha_n(r(s,a) + \gamma \max_{a'\in\mathcal{A}}\widehat Q_{N,n}(s',a') - Q_{n}^{(i)}(s,a)) \label{eq:contradiction:q:l} \\
           &\leq  Q_{n}^{(i)}(s,a) + \alpha_n(r(s,a) + \gamma \max_{a'\in\mc A}  Q^{(i)}_{n}(s',a') - Q_{n}^{(i)}(s,a)), \label{eq:contradiction:q:l:2}
        \end{align}
    \end{subequations}
    for any $s\in\mc S$ and $a\in \mc A$. 
    When considering $S_n=s$, $A_n=a$ and $S_{n+1}=s'$, we see that the upper bound \eqref{eq:contradiction:q:l:2} is Watkins' Q-learning which leads to a bounded Q-function, so $Q_{n+1}^{(i)}(s,a)$ needs to be bounded for all $s,a$ which is a contradiction.
    The case where there is an $s'\in\mc S$ and $a'\in\mc A$ such that $\lim_{n\to\infty}Q_n^{(i)}(s',a') = -\infty$ follows analogously. 
    Consequently, the Q-functions $Q_n^{(i)}$ for any $n$ and any $i$ are bounded and \eqref{eq:var:ass:iiii} indeed holds, which implies Assumption~\ref{item:4:SA}.
\end{proof}
%

\subsection{Estimation Bias}\label{ssec:estimation:bias}

We now focus on the estimation bias of 2RA Q-learning induced by the term $\mathcal{E}_\rho$ in \eqref{eq:robust:estimator}. 
While Watkins' Q-learning suffers from the mentioned overestimation bias, the proposed 2RA Q-learning~\eqref{eq:q-learning:2RA} allows us to control the estimation bias via the parameters $\rho$ and $N$.
We show that for $\rho>0$ with high probability, 2RA Q-learning generates an underestimation bias, somewhat similar to Double Q-learning. 
However, in contrast to Double Q-learning, we can control the level of underestimation via the parameter $\rho$. 
Moreover, the second parameter $N$, describing the number of action-value estimates, further allows us to control the estimation bias.
\begin{theorem}[Estimation bias] \label{thm:estimation:bias} \
    \begin{enumerate}[label=(\roman*)]
        \item \label{asser:i:thm:bias}Consider any $N,n\in\mathbb{N}$, $\rho\geq 0$ and $i\in\{1,\dots,N\}$. If $\mathbb{E}[\theta_n^{(i)}]\in \mc B_\rho(\widehat\theta_{N,n})$, then the robust estimator defined in \eqref{eq:robust:estimator} provides an underestimation where the level of underestimation is controlled by $\rho$, i.e., for all $x\in \mathcal{X}, s'\in\mc S$
        \begin{equation*}
            0 \leq \gamma \phi(x)\max_{a'\in\mc A} \mathbb{E}[\phi(s',a')^\top\theta^{(i)}_n] - \mathbb{E}[\mc E_{\rho}(x,s',\widehat\theta_{N,n})] \leq \sqrt{\rho} \gamma \phi(x)\max_{a'\in\mc A}\|\phi(s',a')\|_2.
        \end{equation*}
        \item \label{asser:ii:thm:bias} Let $\theta^{(i)}$ be initialized with the same value for $i=1,\dots,N$.
        For any $\rho\geq 0$ and $n\in\mathbb{N}$, the robust estimator \eqref{eq:robust:estimator} satisfies for all $x\in\mathcal{X},s'\in\mc S$
        \begin{equation*}
            \lim_{N\to\infty} \mathbb{E}[\mc E_{\rho}(x,s',\widehat\theta_{N,n})] = \gamma \phi(x)\max_{a'\in\mc A} \left\{ \mathbb{E}[\phi(s',a')^\top\theta^{(i)}_n] -\sqrt{\rho}\|\phi(s',a')\|_2 \right\}.
        \end{equation*}
    \end{enumerate}
\end{theorem}
\begin{proof}
    To prove Assertion~\ref{asser:i:thm:bias}, note that according to the definition of the robust estimator \eqref{eq:robust:estimator} for any $\mathbb{E}[\theta_n^{(i)}]\in \mc B_\rho(\widehat\theta_{N,n})$ it must hold that
    \begin{align*}
        \mc E_{\rho}(x,s',\widehat\theta_{N,n}) 
        \leq \gamma \phi(x)\max_{a'\in\mc A} \phi(s',a')^\top\mathbb{E}[\theta^{(i)}_n]
        = \gamma \phi(x)\max_{a'\in\mc A} \mathbb{E}[\phi(s',a')^\top\theta^{(i)}_n] \quad \forall x\in \mathcal{X}, s'\in\mc S,
    \end{align*}
    which implies
    \begin{equation}\label{eq:bias:ass:i:step1}
         \mathbb{E}[\mc E_{\rho}(x,s',\widehat\theta_{N,n})] \leq \gamma \phi(x)\max_{a'\in\mc A} \mathbb{E}[\phi(s',a')^\top\theta^{(i)}_n] \quad \forall x\in \mathcal{X}, s'\in\mc S. 
    \end{equation}
    A lower bound can be derived as for all $x\in\mathcal{X}$, $s'\in\mathcal{S}$
    \begin{subequations} \label{eq:bias:ass:i:step2}
        \begin{align}
           \mathbb{E}[\mc E_{\rho}(x,s',\widehat\theta_{N,n})] 
           &= \gamma \phi(x) \mathbb{E}\left[\max_{a'\in\mc A} \left\{ \phi(s',a')^\top \widehat\theta_{N,n}  - \sqrt{\rho} \|\phi(s',a')\|_2 \right\} \right] \\
           &\geq \gamma \phi(x) \mathbb{E}\left[\max_{a'\in\mc A} \phi(s',a')^\top \widehat\theta_{N,n}\right] - \sqrt{\rho} \gamma \phi(x)\max_{a'\in\mc A}\|\phi(s',a')\|_2 \\ 
           &\geq \gamma \phi(x) \max_{a'\in\mc A} \mathbb{E}\left[\phi(s',a')^\top \widehat\theta_{N,n}\right] - \sqrt{\rho} \gamma \phi(x)\max_{a'\in\mc A}\|\phi(s',a')\|_2 \\
           &=\gamma \phi(x) \max_{a'\in\mc A} \mathbb{E}\left[\phi(s',a')^\top \theta^{(i)}_n\right] - \sqrt{\rho} \gamma \phi(x)\max_{a'\in\mc A}\|\phi(s',a')\|_2,
        \end{align}
    \end{subequations}
    where the first equality is due to Lemma~\ref{lem:computation}. 
    The first inequality follows from splitting the maximization or reverse triangle inequality. 
    The second inequality follows from a Jensen step as explained in \eqref{eq:Jensen:overestimation}. 
    The second equality uses the fact that all $\theta^{(i)}_n$ follow the same distribution.
    Combining \eqref{eq:bias:ass:i:step1} and \eqref{eq:bias:ass:i:step2} implies Assertion~\ref{asser:i:thm:bias}.

    To prove Assertion~\ref{asser:ii:thm:bias}, we first claim that for any $n\in\mathbb{N}$
    \begin{equation}\label{eq:LLN:theta}
        \lim_{N\to\infty} \widehat \theta_{N,n} = \mathbb{E}[\theta_n^{(i)}], 
    \end{equation}
    where the convergence is in mean square, i.e., 
    $\lim_{N\to\infty}\mathbb{E}[\|\widehat\theta_{N,n}-\mathbb{E}[\theta_n^{(i)}]\|^2]=0$.
    This in particular implies that $\widehat \theta_{N,n}$ converges to $\mathbb{E}[\theta^{(i)}]$ in distribution.
    Our second claim states that for any $x\in\mathcal{X}$ and $s'\in\mc S$, the function $\mc E_{\rho}(x,s',\theta')$ is uniformly continuous in $\theta'$. 
    Equipped with these two claims,
    \begin{subequations} \label{eq:steps:ass:ii:bias}
        \begin{align}
            \lim_{N\to\infty} \mathbb{E}[\mc E_{\rho}(x,s',\widehat\theta_{N,n})]
            &=\mathbb{E}[\mc E_{\rho}(x,s',\mathbb{E}[\theta_n^{(i)}])] \\
            &=\mc E_{\rho}(x,s',\mathbb{E}[\theta_n^{(i)}]) \\
            &= \gamma \phi(x)\max_{a'\in\mc A} \left\{ \mathbb{E}[\phi(s',a')^\top\theta^{(i)}_n] -\sqrt{\rho}\|\phi(s',a')\|_2 \right\},
        \end{align}
    \end{subequations}
    where the first equality holds due to the Portmanteau Theorem \citep[Theorem~2.1]{ref:Billingsley-99}, since $\widehat \theta_{N,n}$ converges to $\mathbb{E}[\theta_n^{(i)}]$ in distribution and the function $\mc E_{\rho}(x,s',\theta')$ is uniformly continuous in $\theta'$. 
    The second equality is true as the function $\mc E_{\rho}$ is deterministic, and the last equality is implied by Lemma~\ref{lem:computation}.
    
    It, therefore, remains to show the two claims. 
    To show \eqref{eq:LLN:theta}, we recall that by definition
    \begin{equation*}
        \widehat\theta_{N,n} = \frac{1}{N}\sum_{i=1}^N \theta_n^{(i)}.
    \end{equation*}
    By symmetry of the 2RA Q-learning~\eqref{eq:q-learning:2RA} the variables $\theta_n^{(i)}$ for all $i\in\{1,\dots,N\}$ have the same distribution, but are correlated. 
    We can exploit the weak correlations via the following result.
    \begin{lemma}[{\citet[Theorem~7.1.1]{ref:Brockwell-91}}]\label{lemma:LNN}
        Let $\{Y^{(i)}\}$ be a stationary process with mean $\mu$ and autocovariance function $\gamma(\cdot)$ defined as $\gamma(N) = \mathsf{cov}(Y^{(i+N)},Y^{(i)})$ for any $i\in\mathbb{N}$. Then, 
        \begin{equation*}
           \lim_{N\to\infty} \mathbb{E}\left[\left(\frac{1}{N}\sum_{i=0}^{N-1} Y^{(i)}-\mu\right)^2\right] = 0 \quad \text{if} \quad  \lim_{N\to\infty} \gamma(N)= 0.
        \end{equation*}
    \end{lemma}
    \noindent The 2RA Q-learning~\eqref{eq:q-learning:2RA} is given as
    \begin{equation}\label{eq:update:q:rho:zero}
        \begin{aligned}
            \theta^{(i)}_{n+1} \!&=\! (1\!-\!\alpha_n \beta_n^{(i)}A_1(X_n)) \theta^{(i)}_n\! +\! \gamma \alpha_n \beta_n^{(i)} \phi(X_n) \max_{a'\in\mc A} \left\{\phi(S_{n+1},a')^\top \widehat\theta_{N,n} \!-\!\sqrt{\rho} \|\phi(S_{n+1},a')\|_2\right\} \\
            &\qquad + \alpha_n \beta_n^{(i)} b(X_n), \quad i=1,\dots,N.
        \end{aligned}
    \end{equation}
    We claim that for any $s,t\in\{1,2,\dots,N\}$ and for any $n\in\mathbb{N}_0$
    \begin{equation}\label{eq:covariance:decay}
        \mathsf{cov}(\theta^{(s)}_{n}, \theta^{(t)}_{n}) \leq C\cdot \mathcal{O}(1/N),
    \end{equation}
    where $C\in\mathbb{R}^{d\times d}$ is some constant matrix.
    Then, according to Lemma~\ref{lemma:LNN}, we obtain \eqref{eq:LLN:theta}. 
    It, therefore, remains to show \eqref{eq:covariance:decay}, which we do by induction over $n$. 
    
    The initial condition $\theta^{(i)}_{0}=\theta_0$ is assumed to be some deterministic value for all $i$, then by using the update rule \eqref{eq:update:q:rho:zero} and recalling that $\widehat\theta_{N,0}=\theta_0$
    \begin{align*}
       \mathsf{cov}( \theta^{(s)}_{1}, \theta^{(t)}_{1}) &=   \mathsf{cov}(v + \beta_1^{(s)}w,v + \beta_1^{(t)}w)\\ &=\mathsf{cov}(\beta_1^{(s)}w,\beta_1^{(t)}w) \\
       &=\mathbb{E}[\beta_1^{(s)}\beta_1^{(t)}w w^\top] - \mathbb{E}[\beta_1^{(s)}w]\mathbb{E}[\beta_1^{(t)}w]\\
       &= \mathbb{E}[\beta_1^{(s)}\beta_1^{(t)}]\mathbb{E}[w w^\top] - \mathbb{E}[\beta_1^{(s)}]\mathbb{E}[w]\mathbb{E}[\beta_1^{(s)}]\mathbb{E}[w]^\top \\
       &= - \frac{1}{N^2} \mathbb{E}[w] \mathbb{E}[w]^\top \\
       &\leq C\cdot \mathcal{O}(1/N),
    \end{align*}
    where $v = \theta_0$, $w\!=\!\alpha_n\!\left( -A_1(X_0) \theta_0 \!+\! \gamma \phi(X_0)\max_{a'\in\mc A}\{\phi(S_{1},a')^\top\theta_0\!-\!\sqrt{\rho} \|\phi(S_{1},a')\|_2\} \!+\! b(X_0) \right)$ and we use the fact that $\mathbb{E}[\beta_1^{(s)}\beta_1^{(t)}]=0$ and $\mathbb{E}[\beta_1^{(s)}]=\frac{1}{N}$. 
    Moreover, we use $C=\mathbb{E}[w] \mathbb{E}[w]^\top$.   
    To proceed with the induction step, assume that for any $k,\ell\in\{1,2,\dots,N\}$ we have $\mathsf{cov}(\theta^{(k)}_{n}, \theta^{(\ell)}_{n}) =C\cdot \mathcal{O}(1/N)$ and show that for any $k,\ell\in\{1,2,\dots,N\}$
    \begin{align}\label{eq:induction:step}
        \mathsf{cov}(\theta^{(k)}_{n+1}, \theta^{(\ell)}_{n+1}) \leq C\cdot \mathcal{O}(1/N).
    \end{align}
    Applying the update rule~\eqref{eq:update:q:rho:zero} leads to
    \begin{equation}\label{eq:cov:ind:step}
        \begin{aligned}
        &\mathsf{cov}( \theta^{(k)}_{n+1}, \theta^{(\ell)}_{n+1}) \\
        &\qquad=\mathsf{cov}\bigg((1-\alpha_n \beta_n^{(k)}A_1(X_n)) \theta^{(k)}_n + \beta_n^{(k)}\alpha_n\bigg(\gamma \phi(X_n) \mathcal{M}_{n} +  b(X_n) \bigg)  , \\
           &\qquad \qquad  (1-\alpha_n \beta_n^{(\ell)}A_1(X_n)) \theta^{(\ell)}_n + \beta_n^{(\ell)}\alpha_n\bigg(\gamma \phi(X_n) \mathcal{M}_{n} +  b(X_n) \bigg) \bigg) \\
           &\qquad= \mathsf{cov}\left((1-\alpha_n \beta_n^{(k)}A_1(X_n))\theta^{(k)}_n, (1-\alpha_n \beta_n^{(\ell)}A_1(X_n))\theta^{(\ell)}_n\right) \\
           &\qquad\quad + \mathsf{cov}\bigg((1-\alpha_n \beta_n^{(k)}A_1(X_n))\theta^{(k)}_n, \beta_n^{(\ell)}\alpha_n\left(\gamma \phi(X_n) \mathcal{M}_{n} +  b(X_n) \right) \bigg) \\
            &\qquad\quad + \mathsf{cov}\bigg((1-\alpha_n \beta_n^{(\ell)}A_1(X_n))\theta^{(\ell)}_n, \beta_n^{(k)}\alpha_n\left(\gamma \phi(X_n) \mathcal{M}_{n} +  b(X_n) \right) \bigg)\\
            &\qquad\quad + \mathsf{cov}\bigg(\beta_n^{(k)}\alpha_n\left(\gamma \phi(X_n) \mathcal{M}_{n} +  b(X_n) \right), \beta_n^{(\ell)}\alpha_n\left(\gamma \phi(X_n) \mathcal{M}_{n} +  b(X_n) \right)  \bigg),
        \end{aligned}
    \end{equation}
    where $\mathcal{M}_{n} = \max_{a'\in\mc A}\{\phi(S_{n+1},a')^\top \widehat\theta_{N,n}-\sqrt{\rho} \|\phi(S_{n+1},a')\|_2\}$. We can show that each covariance term from above is of the order $\mathcal{O}(1/N)$. 
    For this, we recall that the properties of $\beta_n$ and its independence with respect to $X_n$ ensure that
    \begin{subequations}\label{eq:subequations:covariance}
        \begin{align}
             \mathsf{cov}(\beta_n^{(\ell)},\beta_n^{(k)}) =\mathcal{O}(1/N) \quad \text{and} \quad \mathsf{cov}(\beta_n^{(\ell)},f(X_n)) = 0
        \end{align}
        for any bounded function $f:\mathcal{X}\to\mathbb{R}$. Moreover, for any $k\neq \ell$
        \begin{equation}
             \begin{aligned}
                \mathsf{cov}(\beta_n^{(\ell)}f(X_n),\beta_n^{(k)}f(X_n)) &= \mathbb{E}[\beta_n^{(\ell)}\beta_n^{(k)}f(X_n) f(X_n)^\top] - \mathbb{E}[\beta_n^{(k)}f(X_n)]\mathbb{E}[\beta_n^{(\ell)}f(X_n)^\top] \\
                &=-\frac{1}{N^2} \|\mathbb{E}[f(X_n)]\|_2^2 \leq \mathcal{O}(1/N).
            \end{aligned}
        \end{equation}
        For any other bounded function $g:\mathcal{X}\to\mathbb{R}$, we get  
        \begin{equation}
            \begin{aligned}
                \mathsf{cov}(\beta_n^{(\ell)}f(X_n),g(X_n)) &= \mathbb{E}[\beta_n^{(\ell)}f(X_n) g(X_n)] - \mathbb{E}[\beta_n^{(\ell)}f(X_n)]\mathbb{E}[g(X_n)] \\
                &=\frac{1}{N}\mathbb{E}[f(X_n) g(X_n)]-\frac{1}{N}\mathbb{E}[f(X_n)]\mathbb{E}[g(X_n)] = \mathcal{O}(1/N).
            \end{aligned}
        \end{equation}
        Similarly, we obtain for any $s,t\in\{1,2,\dots,N\}$
        \begin{align}
            \mathsf{cov}(\theta_n^{(k)},\beta_n^{(\ell)}f(X_n)) = C\cdot\mathcal{O}(1/N),
        \end{align}
        and
        \begin{align*}
            \mathsf{cov}(\theta_n^{(k)},g(X_n) \widehat \theta_{N,n}) = \frac{1}{N}\sum_{\ell=1}^N \mathsf{cov}(\theta_n^{(k)},g(X_n)  \theta^{(\ell)}_{n}),
        \end{align*}
        whereby by using the results of \citet{ref:Bohrnstedt-69}, we can show that 
        $\mathsf{cov}(\theta_n^{(k)},g(X_n)  \theta^{(\ell)}_{n}) =C\cdot \mathcal{O}(1/N)$.
        Therefore, 
        \begin{equation}
            \mathsf{cov}(\theta_n^{(k)},g(X_n) \widehat \theta_{N,n})=C\cdot \mathcal{O}(1/N).
        \end{equation}
    \end{subequations}
    Finally, equipped with \eqref{eq:induction:step} and \eqref{eq:subequations:covariance} by exploiting the results of \citet{ref:Bohrnstedt-69} and continuing with \eqref{eq:cov:ind:step}, we show
    \begin{align*}
        \mathsf{cov}( \theta^{(k)}_{n+1}, \theta^{(\ell)}_{n+1}) = C\cdot\mathcal{O}(1/N),
    \end{align*}
    which completes the induction step. 
    We, therefore, have shown \eqref{eq:covariance:decay} and hence \eqref{eq:LLN:theta}.  
    Regarding our second claim, we show that for any $x\in\mathcal{X}$ and $s'\in\mc S$, the function $\mc E_{\rho}(x,s',\theta)$ is uniformly continuous in $\theta$. 
    For any fixed $x\in\mathcal{X}$ and $s'\in\mc S$ the the function $\mc E_{\rho}(x,s',\theta)$ can be expressed as
    \begin{equation*}
       \mc E_{\rho}(x,s',\theta) = \gamma \phi(x) \max_{a'\in\mc A} \left\{\phi(s',a')^\top\theta-\sqrt{\rho} \|\phi(s',a')\|_2\right\}
       = \gamma \phi(x)  \varphi(\theta, s'),
    \end{equation*}
    where $\varphi(\theta, s') = \max_{a'\in\mc A}\left\{ \phi(s',a')^\top\theta-\sqrt{\rho} \|\phi(s',a')\|_2\right\}$. 
    To show that $\mc E_{\rho}(x,s',\theta)$ is uniformly continuous in $\theta$, it remains to show that $\varphi$ is uniformly continuous in $\theta$. Clearly, $\varphi$ is Lipschtiz continuous in $\theta$, as $\max\{\cdot\} - \max\{\cdot\} \leq \max\{\cdot - \cdot\}$, which is the reversed triangle inequality for the $\infty$-norm. This implies the desired uniform continuity and hence hence $\mc E_{\rho}(x,s',\theta)$ is uniformly continuous in $\theta$.
    Having shown both claims completes the proof.
\end{proof}
Assertion \ref{asser:ii:thm:bias} directly implies that for $\rho=0$, the robust estimator \eqref{eq:robust:estimator} is unbiased in the limit as $N\to\infty$.
We can alternatively show this via Theorem~\ref{thm:estimation:bias}\ref{asser:i:thm:bias}.
By following the proof of Theorem~\ref{thm:estimation:bias}, we can show that for any $\rho>0$
\begin{equation} \label{eq:limit:assertion:1:bias}
    \lim_{N\to\infty} \mathbb{P}\left( \mathbb{E}[\theta_n^{(i)}]\in \mc B_\rho(\widehat\theta_{N,n}) \right) = 1,
\end{equation}
i.e., for any given $\rho$ if $N$ is chosen large enough, we can expect the assumption of Assertion~\ref{asser:i:thm:bias} of Theorem~\ref{thm:estimation:bias} to hold. To derive~\eqref{eq:limit:assertion:1:bias} note that in the proof of Theorem~\ref{thm:estimation:bias}, we show (see \eqref{eq:LLN:theta} and apply Hoelder's inequality) that
    $\lim_{N\to\infty} \mathbb{E}[ \| \widehat\theta_{N,n} - \mathbb{E}[\theta_n^{(i)}] \|] =0.$
Markov's inequality states that for any $\rho>0$
\begin{equation*}
    \mathbb{P}\left( \| \widehat\theta_{N,n} - \mathbb{E}[\theta_n^{(i)}] \| > \sqrt{\rho} \right) \leq \frac{1}{\sqrt{\rho}} \mathbb{E}\left[ \| \widehat\theta_{N,n} - \mathbb{E}[\theta_n^{(i)}] \| \right],
\end{equation*}
which directly implies \eqref{eq:limit:assertion:1:bias}.

\begin{corollary}[Vanishing estimation bias]\label{corol:estimation:bias:vanishing}
    Under the assumptions of Theorem~\ref{thm:convergence} and a regularization sequence $\{\rho_n\}_{n\in\mathbb{N}}\subset \mathbb{R}_+$ such that $\lim_{n\to\infty} \rho_n=0$, for any $N\in\mathbb{N}$ and $i\in\{1,\dots,N\}$
    \begin{equation*}
        \lim_{n\to\infty} \mathbb{E}[\mc E_{\rho_n}(x,s',\widehat\theta_{N,n})] = \lim_{n\to\infty} \gamma \phi(x)\max_{a'\in\mc A} \mathbb{E}[\phi(s',a')^\top\theta^{(i)}_n] \quad \forall x\in\mathcal{X},s'\in\mc S.
    \end{equation*}
\end{corollary}
\begin{proof}
    Recall that according to Theorem~\ref{thm:convergence} for any $i\in\{1,\dots,N\}$ we have $\lim_{n\to\infty}\theta^{(i)}_n = \theta^\star$ almost surely and accordingly $\lim_{n\to\infty}\widehat \theta_{N,n} = \theta^\star$ almost surely. 
    By the definition of $\mc E_{\rho_n}$ in~\eqref{eq:robust:estimator}
    \begin{equation}\label{eq:cor:van:gradient}
        \gamma \phi(x)\max_{a'\in\mc A} \mathbb{E}[\phi(s',a')^\top\theta^\star] = \mathbb{E}[\mc E_{0}(x,s',\theta^\star)] \quad \forall x\in \mathcal{X}, s'\in\mc S.
    \end{equation}
    Therefore, for all $x\in \mathcal{X}, s'\in\mc S$
    \begin{subequations}
        \begin{align}
            \lim_{n\to\infty} \mathbb{E}\left[\mc E_{\rho_n}(x,s',\widehat\theta_{N,n})\right]
            &=\mathbb{E}\left[ \lim_{n\to\infty}\mc E_{\rho_n}(x,s',\widehat\theta_{N,n})\right] \label{eq:fist:DCT}\\
            &\label{eq:second:lemma}=\mathbb{E}\left[ \lim_{n\to\infty} \gamma \phi(x) \max_{a'\in\mc A} \left\{ \phi(s',a')^\top \widehat\theta_{N,n}  - \sqrt{\rho_n} \|\phi(s',a')\|_2 \right\}\right] \\
            &\label{eq:third:interchange}= \mathbb{E}\left[\gamma \phi(x) \max_{a'\in\mc A} \left\{ \phi(s',a')^\top \theta^\star \right\} \right] \\
            &\label{eq:fourth:def:E0}= \mathbb{E}\left[\mc E_{0}(x,s',\theta^\star)\right] \\
            &\label{eq:five:previous}= \gamma \phi(x)\max_{a'\in\mc A} \mathbb{E}[\phi(s',a')^\top\theta^\star] \\
            &\label{eq:six:interchange}= \gamma \phi(x)\max_{a'\in\mc A} \lim_{n\to\infty} \mathbb{E}\left[\phi(s',a')^\top\theta_n^{(i)}\right]\\
            &\label{eq:seven:interchange}= \lim_{n\to\infty}\gamma \phi(x)\max_{a'\in\mc A}  \mathbb{E}\left[\phi(s',a')^\top\theta_n^{(i)}\right],
        \end{align}
    \end{subequations}
    where the equality \eqref{eq:fist:DCT} follows from the bounded convergence theorem. 
    The equality \eqref{eq:second:lemma} is due to Lemma~\ref{lem:computation}. 
    In \eqref{eq:third:interchange}, we use the fact that the limit and maximum can be interchanged as the maximum is over a finite set and that $\widehat \theta_{N,n}$ converges to $\theta^\star$ due to Theorem~\ref{thm:convergence}. 
    The step \eqref{eq:fourth:def:E0} uses the definition of $\mathcal{E}_0$. 
    The equality~\eqref{eq:five:previous} is due to \eqref{eq:cor:van:gradient} and \eqref{eq:six:interchange} uses that $\theta^{(i)}_n$ converges to $\theta^\star$ due to Theorem~\ref{thm:convergence} together with the bounded convergence theorem. 
    Finally, \eqref{eq:seven:interchange} uses again the fact that the limit and maximum can be interchanged as the maximum is over a finite set.
\end{proof}
\noindent Theorem~\ref{thm:estimation:bias} allows us to interpret the choice of regularization $\{\rho_n\}_{n\in\mathbb{N}}$ in a non-asymptotic manner.

\begin{remark}[Selection of parameter $\rho_n$]
    We have shown in Theorem~\ref{thm:convergence} that convergence of 2RA Q-learning~\eqref{eq:q-learning:2RA} requires a sequence $\{\rho_n\}_{n\in\mathbb{N}}$ such that $\lim_{n\to\infty}\rho_n=0$. Theorem~\ref{thm:estimation:bias} provides insights into how the specific decay of $\rho_n$ determines the resulting performance of 2RA Q-learning. 
    More precisely, Theorem~\ref{thm:estimation:bias} describes the inherent trade-off in the selection of $\rho_n$: choosing larger values of $\rho_n$ increases the probability that $\mathbb{E}[\theta_n^{(i)}]\in\mathcal{B}_{\rho_n}(\widehat \theta_{N,n})$ which guarantees an underestimation bias. 
    On the other hand, choosing smaller values of $\rho_n$ decrease the level of underestimation bias but potentially introduce an overestimation bias as the probability that $\mathbb{E}[\theta_n^{(i)}]\notin\mathcal{B}_{\rho_n}(\widehat \theta_{N,n})$ increases. 
    We further comment on the choice of regularization $\rho_n$ in the numerical experiments, Section~\ref{sec:numerical:results}.   
\end{remark}
\begin{remark}[Selection of parameter $N$]
    The convergence of 2RA Q-learning holds for any choice of $N$; see Theorem~\ref{thm:convergence}. 
    Moreover, Theorem~\ref{thm:estimation:bias} states that increasing $N$ decreases the estimation bias.
    Choosing the parameter $N$ too large, however, when using a learning rate that is $N$-times the learning rate of Watkins' Q-learning (according to Theorem~\ref{thm:AMSE}) can lead to numerical instability. Therefore, in practice, a trade-off must be made when selecting $N$.
\end{remark}

\subsection{Asymptotic Mean-Squared Error}\label{ssec:AMSE}

We have shown in Theorem~\ref{thm:convergence} that 2RA Q-learning~\eqref{eq:q-learning:2RA} asymptotically converges to the optimal Q-function. 
This section investigates the convergence rate via the so-called asymptotic mean-squared error. 
Throughout this section, we consider a tabular setting and assume without loss of generality that the optimal Q-function is such that $\theta^\star=0$. 
If $\theta^\star\neq 0$, the results can hold by subtracting $\theta^\star$ from the estimators of the Q-learning, see \citet{ref:Devraj-17}.
Given the 2RA Q-learning and the corresponding estimator $\widehat \theta_{N,n}$ as introduced in \eqref{eq:estimator:proposed:DRO}, we define its asymptotic mean-squared error (AMSE) as the limit of a scaled covariance 
\begin{equation*}
    \mathsf{AMSE}(\widehat \theta_{N}) = \lim_{n\to\infty} n \mathbb{E}[\widehat \theta_{N,n}^\top \widehat \theta_{N,n}]
    =\lim_{n\to\infty} n \mathbb{E}[\|\widehat \theta_{N,n} \|_2^2].
\end{equation*}
Our analysis also discusses the choice of the learning rate in 2RA Q-learning compared to the learning rate of Watkins' Q-learning.
\begin{theorem}[AMSE for 2RA Q-learning] \label{thm:AMSE}
    Consider a setting where the regularization sequence $\{\rho_n\}_{n\in\mathbb{N}}\subset \mathbb{R}_+$ is such that $\lim_{n\to\infty}\rho_n=0$ and $N\in\mathbb{N}$. 
    Let $\alpha_n^{\mathsf{QL}}=g/n$ be the learning rate of Watkins' Q-learning~\eqref{eq:q-learning} and consider the 2RA Q-learning~\eqref{eq:q-learning:2RA} with learning rate $\alpha_n = N\cdot g/n$, where $g$ is a positive constant\footnote{We assume that our starting index for $n$ is large enough such that $Ng/n<1$.}. 
    Then, there exists some $g_0>0$ such that for any $g>g_0$
    \begin{equation*}
        \mathsf{AMSE}(\widehat \theta_{N}) = \mathsf{AMSE}(\theta^{\mathsf{QL}}),
    \end{equation*}
    where $\{\theta^{\mathsf{QL}}_n\}_{n\in\mathbb{N}}$ is a sequence generated by Watkins' Q-learning algorithm.
\end{theorem}
\begin{proof}
    Our proof is inspired by the recent treatment of Double Q-learning \citep{ref:Wentao-20} and by \citet{ref:Devraj-17} analyzing the asymptotic properties of Q-learning. 
    The key idea is to recall that from the proof of Theorem~\ref{thm:convergence}, we know that $\theta^{(i)}_{n}\to \theta^\star = 0$ as $n\to\infty$ for any $i\in\{1,2,\dots,N\}$. 
    Hence, we can express the AMSE of $\theta^{(i)}$ alternatively as
    \begin{equation*}
        \mathsf{AMSE}(\theta^{(i)}) =  \lim_{n\to\infty} n \mathbb{E}[{\theta_n^{(i)}}^\top \theta_n^{(i)}]
        =  \tr{ \lim_{n\to\infty} n \mathbb{E}[\theta_n^{(i)} {\theta_n^{(i)}}^\top]}
        =\tr{V},
    \end{equation*}
    where the matrix $V= \lim_{n\to\infty} n \mathbb{E}[\theta_n^{(i)} {\theta_n^{(i)}}^\top]$ is called the \textit{asymptotic covariance}.
    It has been shown in \citet{ref:Devraj-17} that the asymptotic covariance of Watkins' Q-learning~\eqref{eq:q-learning} can be studied via the linearized counterpart, given as
    \begin{equation*}
        \theta^{\mathsf{QL}}_{n+1} = \theta^{\mathsf{QL}}_n + \alpha^{\mathsf{QL}}_n\phi(X_n) (r(X_n) + \gamma \phi(S_{n+1},\pi^\star(S_{n+1}))^\top \theta^{\mathsf{QL}}_n - \phi(X_n)^\top \theta^{\mathsf{QL}}_n),
    \end{equation*}
    where $\pi^\star$ is the optimal policy based on $\theta^\star$. 
    Using similar arguments from \citet{ref:Devraj-17} and \citet{ref:Wentao-20}, we can show that the asymptotic variance of 2RA Q-learning~\eqref{eq:q-learning:2RA}, which is defined as $\lim_{n\to\infty} n \mathbb{E}[\widehat \theta_{N,n}\widehat \theta_{N,n}^\top]$, can be studied by considering the linearized recursion with $\rho_n=0$, given as
    \begin{equation}\label{eq:asymptotic:DRO:qlearning}
        \begin{aligned}
            \theta^{(i)}_{n+1} &= \theta^{(i)}_n + \alpha_n \beta_n^{(i)}(b(X_n) - A_1(X_n)\theta^{(i)}_n + \underbrace{\gamma \phi(X_n)  \phi(S_{n+1},\pi^\star(S_{n+1}))^\top}_{=A_2(Z_n)}\widehat \theta_{N,n}), \quad i=1,\dots, N,
        \end{aligned}
    \end{equation}
    where $Z_n=(X_n,S_{n+1})$, $b(X_n) = \phi(X_n)r(X_n)$ and $A_1(X_n) = \phi(X_n) \phi(X_n)^\top$.
    We formally justify this linearization argument in Lemma~\ref{lem:linearizaton}.
    Using a more compact notation where 
    $\theta_n = (\theta^{(1)\top}_{n}, \dots, 
    \theta^{(N)\top}_{n})^\top$ and $\beta_n = (\beta_n^{(1)}, \dots, \beta_n^{(N)})^\top$ and choosing the learning rate $\alpha_n = \alpha_n^{\mathsf{QL}}\cdot N$, the update equation \eqref{eq:asymptotic:DRO:qlearning} can be expressed in a standard form as
    \begin{equation}\label{eq:system:SA}
        \theta_{n+1} = \theta_n + \alpha_n^{\mathsf{QL}} (\mathsf{b}(X_n) + \mathsf{A}_2(Z_n)\theta_n  - \mathsf{A}_1(X_n)\theta_n),
    \end{equation}
    where 
    \begin{equation}\label{eq:symbol:A1:A2:b}
        \begin{aligned}
            \mathsf{A}_1(X_n) &= N \cdot \mathsf{diag}(\beta_n^{(1)}A_1(X_n),\dots,\beta_n^{(N)}A_1(X_n)), \\
            \mathsf{A}_2(Z_n) &= 
            \begin{pmatrix}
                \beta_n^{(1)}  A_2(Z_n) & \dots & \beta_n^{(1)} A_2(Z_n) \\
                \beta_n^{(2)}  A_2(Z_n) & \dots & \beta_n^{(2)}  A_2(Z_n) \\
                \vdots & \ddots & \vdots \\
                \beta_n^{(N)}  A_2(Z_n) & \dots  & \beta_n^{(N)} A_2(Z_n)
            \end{pmatrix},\qquad
            \mathsf{b}(X_n) = 
            \begin{pmatrix}
                N \beta_n^{(1)} \cdot b(X_n)\\
                \vdots \\
                N \beta_n^{(N)} \cdot b(X_n)\\
            \end{pmatrix}.
        \end{aligned}
    \end{equation}
    Let $\mu$ denote the invariant distribution of the Markov chain $\{X_n\}_{n\in\mathbb{N}}$ and let $D$ be a diagonal matrix with entries $D_{ii}= \mu_i$ for all $i=1,\dots,|\mc X|$. 
    Then, when considering $X_\infty$ as a random variable under the stationary distribution, we introduce
    \begin{equation}\label{eq:bar:symbol:A1:A2:b}
        \begin{aligned}
            &\bar A_1 = \mathbb{E}[A_1(X_\infty)]=\Phi D\Phi^\top, \quad &&\bar A_2 =  \mathbb{E}[A_2(Z_\infty)]= \gamma \Phi D P S_{\pi^\star} \Phi^\top, \\ 
            &\bar{\mathsf{A}}_1=\mathbb{E}[\mathsf{A}_1(X_\infty)],  &&\bar{\mathsf{A}}_2=\mathbb{E}[\mathsf{A}_2(Z_\infty)],
        \end{aligned}
    \end{equation}
    where $S_{\pi^\star}$ is the action selection matrix of the optimal policy $\pi^\star$ such that $S_{\pi^\star}(s,(s,\pi^\star(s)))=1$ for $s\in\mc S$ and $\Phi$ is defined in \eqref{eq:q:approx}.
    With these variables at hand, we define $\bar{\mathsf{A}} = \bar{\mathsf{A}}_2 - \bar{\mathsf{A}}_1$, i.e.,
    \begin{align*}
        \bar{\mathsf{A}}=
        \begin{pmatrix}
            \frac{1}{N} \bar A_2-\bar A_1 & \frac{1}{N} \bar A_2& \frac{1}{N} \bar A_2 & \dots & \frac{1}{N} \bar A_2 \\
            \frac{1}{N} \bar A_2 & \frac{1}{N} \bar A_2-\bar A_1 & \frac{1}{N} \bar A_2 & \dots & \frac{1}{N} \bar A_2\\
            \vdots & \vdots & \vdots & \ddots & \vdots \\
            \frac{1}{N} \bar A_2 & \frac{1}{N} \bar A_2  &\frac{1}{N} \bar A_2 & \dots & \frac{1}{N} \bar A_2-\bar A_1
        \end{pmatrix}.
    \end{align*}
    Moreover, we introduce
    \begin{align*}
        \Sigma_{\mathsf{b}} = \mathbb{E}[\mathsf{b}(X_0)\mathsf{b}(X_0)^\top] + \sum_{n\geq 1} \mathbb{E}[\mathsf{b}(X_n)\mathsf{b}(X_0)^\top+\mathsf{b}(X_0)\mathsf{b}(X_n)^\top].
    \end{align*}
    According to the definition of $\mathsf{b}$, we get the block-diagonal structure
    \begin{align*}
        \mathbb{E}[\mathsf{b}(X_0) \mathsf{b}(X_0)^\top] = N \cdot \mathsf{diag}( \mathbb{E}[b(X_0) b(X_0)^\top], \dots, \mathbb{E}[b(X_0) b(X_0)^\top]),
    \end{align*}
    where the expectation is in steady-state.
    Moreover,
    \begin{align*}
      \mathbb{E}[\mathsf{b}(X_n) \mathsf{b}(X_0)^\top] = 
      \begin{pmatrix}
          \mathbb{E}[b(X_n) b(X_0)^\top] & \dots & \mathbb{E}[b(X_n) b(X_0)^\top] \\
          \vdots & & \vdots \\
          \mathbb{E}[b(X_n) b(X_0)^\top] & \dots & \mathbb{E}[b(X_n) b(X_0)^\top]
      \end{pmatrix},
    \end{align*}
    which eventually leads to a matrix of the form
    \begin{align}\label{eq:asympt:covariance}
        \Sigma_{\mathsf{b}}\! =\! 
        \begin{pmatrix}
            N \mathbb{E}[b(X_0) b(X_0)^\top] \!+\! 2 B_2 & 2B_2 & \dots & 2B_2 \\
            2 B_2 & N \mathbb{E}[b(X_0) b(X_0)^\top] \!+\! 2 B_2 & \dots & 2B_2 \\
            \vdots & \vdots & \ddots & \vdots \\
            2 B_2 & 2B_2 &  \dots & N \mathbb{E}[b(X_0) b(X_0)^\top] \!+\! 2 B_2
        \end{pmatrix}\!,
    \end{align}
    where we introduce $B_2 = \frac{1}{2}\sum_{n\geq 1} \mathbb{E}[b(X_n)b(X_0)^\top+b(X_0)b(X_n)^\top]$ and $B_1 = \mathbb{E}[b(X_0)b(X_0)^\top]+ B_2$.
    
    We define $g_0 = \inf\{g\geq 0 \ : \ g \max\{\lambda_{\max}(\bar A),\lambda_{\max}(\bar{\mathsf{A}})\} < -1\}$ and note that $g_0$ exists, since both $\bar A$ and $\bar{\mathsf{A}}$ are Hurwitz as the corresponding Q-learning variants converge (Theorem~\ref{thm:convergence}). 
    As a result for any $g>g_0$ the matrix $\frac{1}{2}I + g \bar{\mathsf{A}}$ is Hurwitz and hence the Lyapunov equation
    \begin{equation}\label{eq:lyapounov:robust:ql}
        \Sigma_\infty \left( \frac{1}{2}I + g \bar{\mathsf{A}}^\top\right) + \left( \frac{1}{2}I + g \bar{\mathsf{A}}\right)\Sigma_\infty + g^2 \Sigma_{\mathsf{b}} = 0
    \end{equation}
    has a unique solution, which describes AMSE of our proposed method (see \citet{ref:Busic-20} and \citet[Theorem~1]{ref:Wentao-20}), i.e., $\Sigma_\infty = \lim_{n\to\infty} n \mathbb{E}[\theta_n \theta_n^\top]$.
    Due to the symmetry of the proposed scheme, the matrix $\Sigma_\infty$ will consist of diagonal elements equal to $V = \lim_{n\to\infty}n\mathbb{E}[\theta_n^{(i)} \theta_n^{(i)\top}]$ and off-diagonal entries $C = \lim_{n\to\infty}n\mathbb{E}[\theta_n^{(i)} \theta_n^{(j)\top}]$ for $i\neq j$. 
    Therefore, summing the first row of matrices in \eqref{eq:lyapounov:robust:ql} and using \eqref{eq:asympt:covariance} leads to
    \begin{equation}\label{eq:lyap:components}
        V + (N-1) C + g (V + (N-1)C) (\bar A_2 - \bar A_1)^\top + g(\bar A_2 - \bar A_1)(V + (N-1)C) + g^2 N(B_1 + B_2)=0.
    \end{equation}

    Due to the definition of $g_0$, the matrix $\frac{1}{2}I + g \bar{{A}}$ is Hurwitz and hence the Lyapunov equation 
    \begin{equation}\label{eq:lyapounov:ql}
        \Sigma_\infty^{\mathsf{QL}} \left( \frac{1}{2}I + g (\bar{A}_2 - \bar A_1)^\top\right) + \left( \frac{1}{2}I + g (\bar{A}_2 - \bar A_1)\right)\Sigma_\infty^{\mathsf{QL}} + g^2 \left( B_1 + B_2 \right) = 0
    \end{equation}
    has a unique solution, which is denoted by $\Sigma_\infty^{\mathsf{QL}}$ and describes the AMSE of Watkins' Q-learning, i.e., $\Sigma_\infty^{\mathsf{QL}} = \lim_{n\to\infty}\mathbb{E}[\theta_n^{\mathsf{QL}} \theta_n^{\mathsf{QL}\top}]$, see \citet{ref:Wentao-20}.
    
    By comparing \eqref{eq:lyapounov:ql} with \eqref{eq:lyap:components} and recalling that the solutions are unique, we obtain $\Sigma_\infty^{\mathsf{QL}} = \frac{V + (N-1) C}{N}$. 
    Finally,
    \begin{subequations}
        \begin{align}
            \mathsf{AMSE}(\widehat \theta_{N}) &= \lim_{n\to\infty} n \mathbb{E}\left[ \left(\frac{1}{N}\sum_{j=1}^N \theta_n^{(j)}\right)^\top \left(\frac{1}{N}\sum_{i=1}^N \theta_n^{(i)}\right) \right] \\
            &= \frac{1}{N^2} \lim_{n\to\infty} n \mathbb{E}\left[ \left(\sum_{j=1}^N \theta_n^{(j)\top}\right) \left(\sum_{i=1}^N \theta_n^{(i)}\right) \right] \\
            &=\frac{1}{N^2} \left( N\tr{V} + N(N-1)\tr{C}\right) \\
            &=\tr{\frac{V+(N-1)C}{N}} \\
            &=\tr{\Sigma_\infty^{\mathsf{QL}}}\\
            &=\mathsf{AMSE}(\theta^{\mathsf{QL}}) .
        \end{align}
    \end{subequations}
\end{proof}
\begin{remark}[Assumption on learning rate]
    As pointed out by \citet{ref:Wentao-20} the condition $g>g_0$ in the tabular setting reduces to $g>\frac{1}{\mu_{\min}(1-\gamma)}$, where $\mu_{\min}$ denotes the minimum entry of the stationary distribution of the state.
\end{remark}

\begin{figure*}[t]
    \begin{subfigure}[b]{0.49\textwidth}
        \centering
        \includegraphics[width=0.75\textwidth]{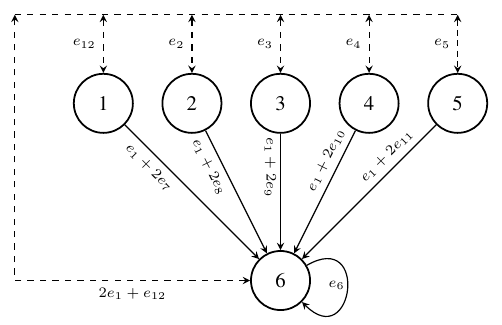}
        \subcaption{Baird's Example\label{fig:baird_env}}
    \end{subfigure}
    \hfill
    \begin{subfigure}[b]{0.5\textwidth}
        \includegraphics[width=\textwidth]{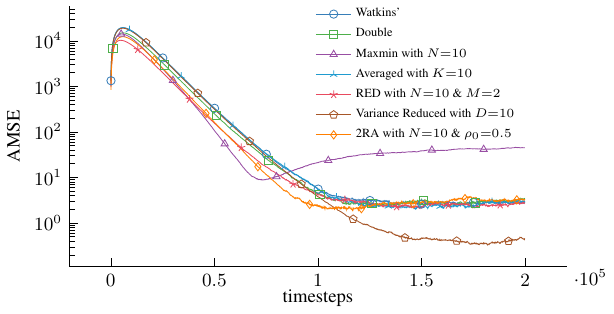}
        \subcaption{Q-learning variants\label{fig:baird_005_compare}}
    \end{subfigure}
    \\[-1mm]
    \begin{subfigure}[b]{0.49\textwidth}
        \includegraphics[width=\textwidth]{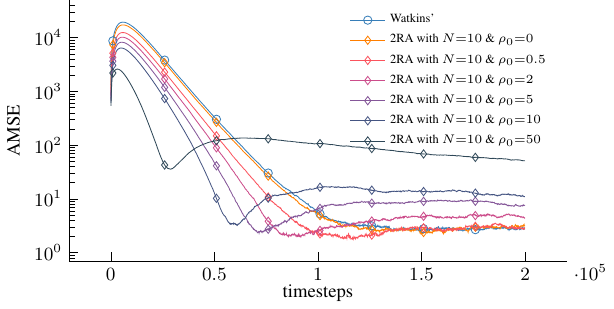}
        \subcaption{Choice of $\rho_{0}$\label{fig:baird_005_rhos}}
    \end{subfigure}
    \hfill
    \begin{subfigure}[b]{0.49\textwidth}
        \includegraphics[width=\textwidth]{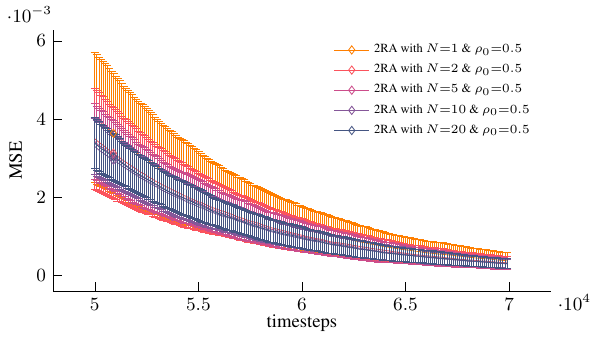}
        \subcaption{Choice of $N$\label{fig:baird_005_Ns}}
    \end{subfigure}
    \caption{Baird's Example. All Methods use an initial learning rate of $\alpha_0=0.01$, $w_{\alpha} = 10^5$, and $\gamma = 0.8$. All 2RA agents additionally use $w_{\rho}=10^{3}$. The reward function has values random-uniformly sampled from $[-0.05, 0.05]$. All results are average over $100$ consecutive experiments. (a)~Baird's example environment with the feature vectors for each state-action pair. (b)~Comparison of the AMSE of Watkins Q-learning, Double Q-learning, Maxmin Q-learning with $N=10$, where the 2RA Q-learning uses initial $\rho_0=0.5$ and $N=10$. (c)~Comparison of the AMSE of 2RA Q-learning with $N=10$ but different initial values $\rho_0$. (d)~Experiment showing the MSE in terms of mean and standard deviation for different values of $N$ with $\rho_0=0.5$.}
    \label{fig:baird_005_figures}
\end{figure*}

\section{Numerical Results}\label{sec:numerical:results}

We numerically\footnote{Here: \href{https://github.com/2RAQ/code}{github.com/2RAQ/code}} compare our presented 2RA Q-learning~\eqref{eq:q-learning:2RA} with Watkins' Q-learning~\eqref{eq:Q:learning:watkins}, Hasselt's Double Q-learning \citep{ref:Hasselt-10}, and with the Maxmin Q-learning \citep{ref:White-20}.
First, we look at Baird's Example \citep{ref:Baird-95}, then we consider arbitrary, randomly generated MDP environments with fixed rewards, and last, the CartPole example~\citep{ref:Barto-83, gym2016}.
In all experiments\footnote{Except for REDQ, since the learning rate would be too high for the multiple updates per step.}, we choose a step size $\alpha_{n} = \frac{N\alpha_{0}w_{\alpha}}{n+w_{\alpha}}$, where $N$ is the number of state-action estimates used in the respective learning method, $w_{\alpha}>0$ is a weight parameter, and $\alpha_{0}$ is the initial step size. 
The decay rate for the regularization parameter $\rho_{n}$, as required for convergence (see Theorem \ref{thm:convergence}), is chosen to be either $\rho_{n}{=}\frac{\rho_{0}w_{\rho}}{n + w_{\rho}}$ or $\rho_{n}{=}\frac{\rho_{0}w_{\rho}}{n^{2} + w_{\rho}}$ with exact parameters given for each experiment and a more detailed evaluation at the end of this section.

\vspace{-2mm}
\paragraph{Baird's Example.}

We consider the setting described in \citet{ref:Wentao-20}.
The environment (Figure~\ref{fig:baird_env}) has six states and two actions. Under action one, the transition probability to any of the six states is $1/6$, and action two results in a deterministic transition to state six. 
These transition dynamics are independent of the state at which an action is chosen.
Therefore the trajectories on which the methods are updated can be obtained from a random uniform behavioral policy that allows every state to be visited.
The features vectors $\phi(s, a)$ are constructed as shown in Figure~\ref{fig:baird_env} where each $e_{i} \in \mathbb{R}^{12}$ is the $i^\text{th}$ unit vector.
In this setting, it is known that the optimal policy is unique \citep{ref:Wentao-20}, and our theoretical results apply. 
All $\theta^{(i)}$ are initialized, as in \citet{ref:Wentao-20}, uniformly at random with values in $[0, 2]$.
Figures~\ref{fig:baird_005_compare} and \ref{fig:baird_005_rhos} have a log-scaled y-axis to emphasize the smaller differences between models as they converge.
The first important observation is that all learning methods converge to the same AMSE, which is in line with Theorem~\ref{thm:AMSE}.
An exception is Maxmin Q-learning, for which, however, no theoretical statement regarding the expected behavior of its AMSE is made.
Higher values for $\rho$ increase the convergence speed in the early learning phase, as shown in Figure~\ref{fig:baird_005_rhos}.
However, if $\rho$ is too large or its relative decay too slow, learning eventually slows down (or even temporarily worsens) as large values of $\rho$ make the update steps too big.
For the proposed choice of $\rho$, our 2RA-method outperforms the other learning methods in the first part of the learning process without getting significantly slower in the long run.
Only the AMSE of Variance Reduced Q-Learning outperforms all other methods which appears to be caused by the specific instance of Baird's experiment (compare results of the Random Environment).
Our next experiment shows that 2RA and Maxmin Q-Learning are sensitive towards different environments which prefer over-, under-, or no estimation bias at all.
Figure~\ref{fig:baird_005_Ns} uses a non-log-scaled y-axis to ensure the size of the standard errors is comparable.
It can be observed how increasing the parameter $N$ reduces the standard error of the learning across multiple experiments.
However, it is also apparent that the marginal utility of each additional theta decreases fast.

\vspace{-2mm}
\paragraph{Random Environment.} \label{ssec:random:environment}

\begin{figure*}[t]
    \begin{subfigure}[c]{0.32\textwidth}
        \includegraphics[width=\textwidth]{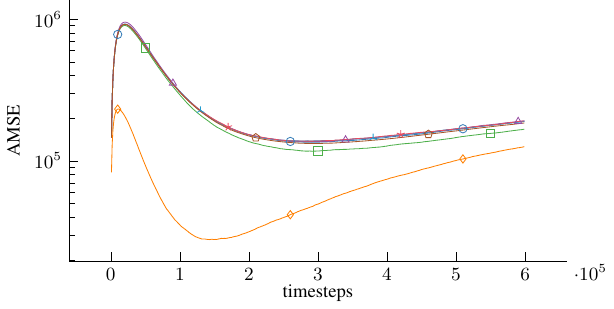}
        \subcaption{\label{fig:random_1}}
    \end{subfigure}
    \hfill
    \begin{subfigure}[c]{0.32\textwidth}
        \includegraphics[width=\textwidth]{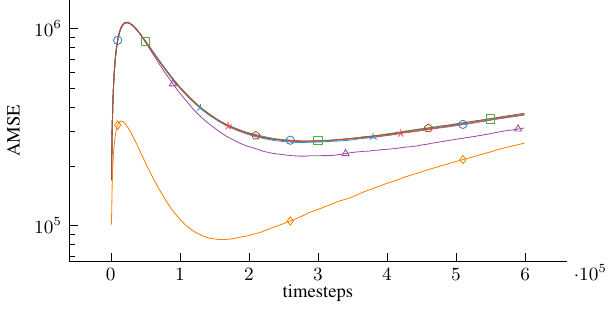}
        \subcaption{\label{fig:random_2}}
    \end{subfigure}
    \hfill
    \begin{subfigure}[c]{0.32\textwidth}
        \includegraphics[width=\textwidth]{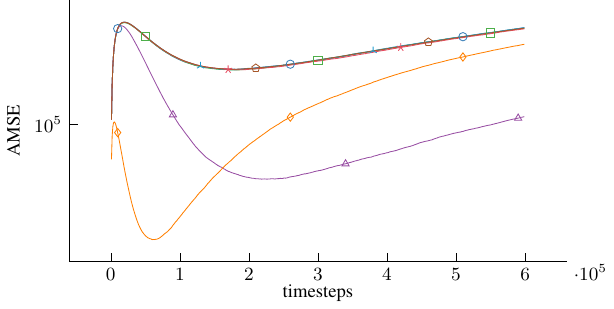}
        \subcaption{\label{fig:random_3}}
    \end{subfigure}
    \\[-1mm]
    \begin{subfigure}[c]{0.32\textwidth}
        \includegraphics[width=\textwidth]{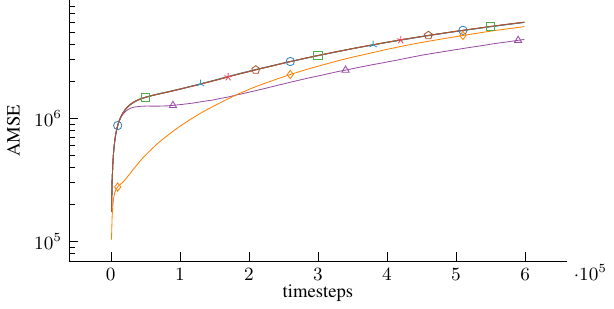}
        \subcaption{\label{fig:random_4}}
    \end{subfigure}
    \hfill
    \begin{subfigure}[c]{0.32\textwidth}
        \includegraphics[width=\textwidth]{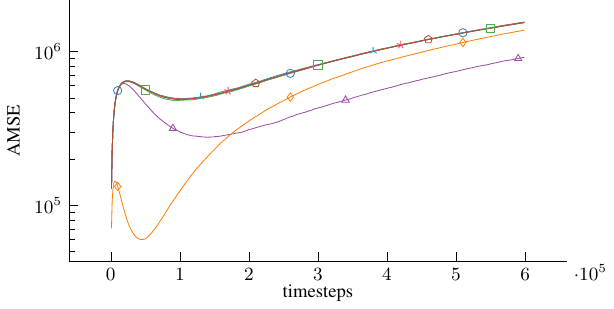}
        \subcaption{\label{fig:random_5}}
    \end{subfigure}
    \hfill
    \begin{subfigure}[c]{0.32\textwidth}
        \includegraphics[width=\textwidth]{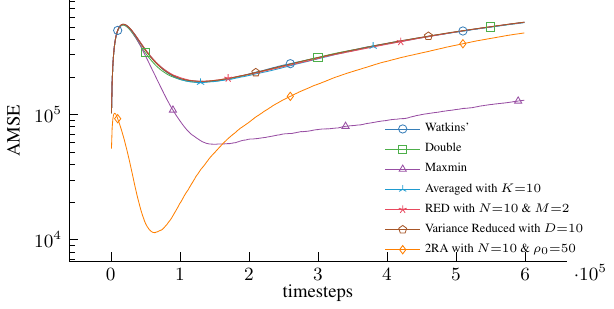}
        \subcaption{\label{fig:random_6}}
    \end{subfigure}
    \caption{Random Environment. All methods use an initial learning rate of $\alpha_0=0.01$, $w_{\alpha}=10^{5}$, $\gamma = 0.9$, and all $\theta^{(i)}$ initialized as zero. Maxmin as well as 2RA Q-learning have $N=10$ and 2RA agents additionally use $\rho_{0}=50$ and $w_{\rho}=10^{4}$. The plots show the first six randomly drawn environments and all results are average over $100$ consecutive experiments. A broader plot of the first 20 random environments is provided in Figure~\ref{fig:appednix_random_env_figures} in Appendix~\ref{app:additional:experiments}.}
    \label{fig:random_env_figures}
\end{figure*}

This experiment visualizes how different learning methods, with a fixed set of hyperparameters, behave under changes in the environment's transition dynamics.
For $|\mc A|=3$ and $|\mc S|=10$, we consider a random environment that is described by a transition probability matrix, which, for each pair $s\in\mathcal{S}$ and $a\in\mathcal{A}$, is drawn from a Dirichlet distribution with uniform parameter $0.1$. 
Analogously, we draw a distribution of the initial states $s_0$. 
Similar to Baird's example, these MDPs are ergodic and random uniform behavioral policies can be used to generate trajectories based on which updates are performed.
We further consider a quadratic reward function $r(s, a) = -qs^{2} - pa^2$ for all possible environments, where $p, q \in \mathbb{R}_+$ are such that $p < q$.
Therefore, different environments have the same reward function but different transition dynamics.
For our experiments, we chose $q = 0.1$ and $p = 0.01$.
Each environment of Figure~\ref{fig:random_env_figures} is randomly drawn, in sequence, from the same random seed.
The resulting dynamics vary significantly between different environments as only one constant selection of hyperparameters is used, but 2RA Q-learning consistently outperforms the other methods in the early stages of learning. \\
With the exception of Maxmin, the other benchmarks perfom similarly to Watkins' Q-Learning.
A further observation is that the better 2RA Q-learning performs in an environment, the better Double Q-learning performs. This indicates that these are environments where an underestimation bias is beneficial \citep{ref:White-20}, with the strength of the effect varying with the drawn transition dynamics.

\vspace{-2mm}
\paragraph{CartPole.}

The well-known CartPole environment \citep{gym2016} serves as a more practical application while still using the same linear function approximation model from the previous experiments in combination with a discretized CartPole state space.
For each timestep, in which the agent can keep the pole within an allowed deviation from an upright angle and the cart's starting position along the horizontal axis, it receives a reward of $+1$.
An episode ends if either one of the thresholds for allowed deviation is broken.
Since a random uniform behavioral policy would not enable visits to all regions of the state space, the latest updated policy combined with $\epsilon$-greedy exploration is used to generate the next timestep which is then used to update the model; Updates are applied after each timestep.
We compare the different learning algorithms based on how many training episodes are required to solve the CartPole task.
The task is considered to be solved if, during the evaluation, the average reward over $100$ episodes with a maximum allowed stepcount of $210$ reaches or exceeds $195$.
Across 1000 experiments, the number of episodes until the task is solved (hit times) is collected for each learning method.
Methods that, on average, solve the environment with fewer training episodes are ranked higher in the performance comparison.
As CartPole benefits from a high learning rate, an initial $\alpha_{0}=0.4$ is chosen and decayed per episode $e$, as compared to the decay per timestep of the previous experiments, such that $\alpha_{e} = \alpha_{0}\frac{w_{\alpha}}{e + w_{\alpha}}$. 

Comparing the hit time distributions of the different algorithms shows that the 2RA mean performance is better than Double Q-learning, which outperforms both Watkins' and Maxmin Q-learning by a significant margin.
CartPole appears to benefit from the underestimation bias introduced by Double and 2RA Q-learning.
This is consistent with the previous experiments where the good performance of 2RA Q-learning correlates with good performance of Double Q-learning.
Since this experiment has deterministically initialized $\theta_{0}$ as well as deterministic state transitions, REDQ and Variance Reduced Q-Learning are not comparable in this setting.

In Appendix~\ref{app:ssec:neural:networks}, we provide an additional example, where we test 2RA Q-learning when used with neural network Q-function approximation, applied to the LunarLander environment \citep{gym2016}. Also there, 2RA Q-learning shows good performance, despite the fact, that our theoretical results do not apply.

\begin{figure*}[t]
    \begin{subfigure}[b]{0.49\textwidth}
        \includegraphics[width=\textwidth]{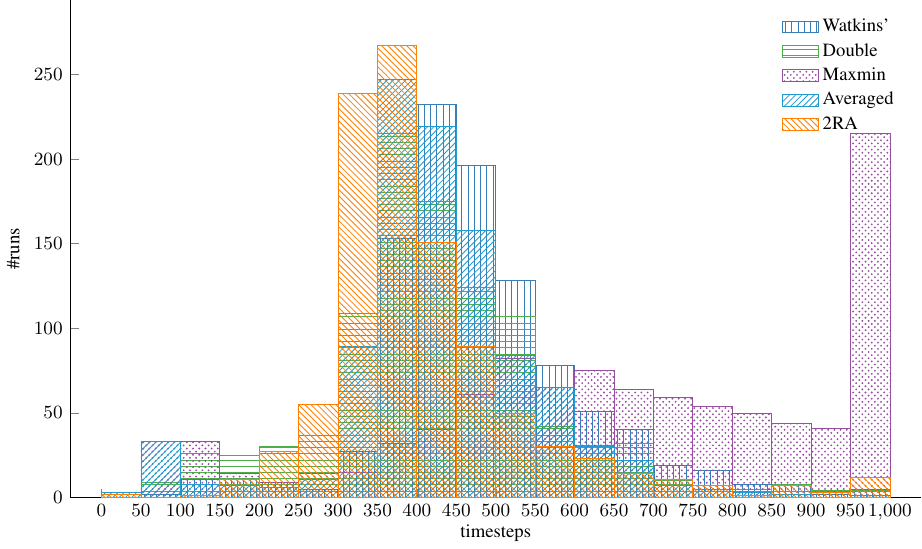}
        \subcaption{Distribution of hit times\label{fig:cartpole_hist}}
    \end{subfigure}
    \hfill
    \begin{subfigure}[b]{0.49\textwidth}
        \centering
        \resizebox{0.75\textwidth}{!}{\begin{tabular}{ |l|c| } 
     \hline
     Algorithm & Mean hit time \\ 
     \hline
     Watkins' Q-Learning & $457.35 \pm 128.18$ \\ 
     Double Q-Learning & $401.89 \pm 144.43$ \\ 
     Maxmin Q-Learning & $645.02 \pm 270.01$ \\ 
     Averaged Q-Learning & $404.09 \pm 124.19$ \\ 
     2RA Q-Learning & $386.19 \pm 133.47$ \\ 
     \hline
\end{tabular}
}\\
        \vspace*{1.5cm}
        \subcaption{Mean and std of hit times\label{fig:cartpole_table}}
    \end{subfigure}
    \caption{Cartpole, 1000 experiments. All methods use an initial learning rate of $\alpha_{0}=0.4$, $w_{\alpha}=100$, $\gamma=0.999$ and all $\theta^{(i)}$ initialized as zero. Maxmin, as well as 2RA Q-learning, have $N=8$. 2RA further uses $\rho_{0}=150$ and $w_{\rho}=10^{4}$. All algorithms are evaluated after every $50$ episodes and recorded if the average evaluation reward reaches or exceeds $195$. (a) Shows the distributions of each algorithm's hit times and (b) lists the respective mean hit times and corresponding standard deviations.}
    \label{fig:cartpole_figures}
\end{figure*}
\section{Discussion and Conclusion} \label{sec:conclusion}

In this work, we proposed 2RA Q-learning and showed that it enables control of the estimation bias via the parameters $N$ and $\rho$ while maintaining the same asymptotic convergence guarantees as Double and Watkin's Q-learning. 
In practice, the control of the estimation bias enables faster convergence to a good-performing policy in finitely many steps which is caused by the intrinsic property of environments to favor an over-, an under-, or no estimation bias at all.
Therefore, determining the optimal bias adjustment is highly dependent on the specific environment and rigorous analysis of environments' bias preferences is not yet available.
To account for this, 2RA Q-learning provides two additional tuning parameters that can be used to fine-tune learning for these environment preferences. 
This level of control, combined with computational costs comparable to existing methods, makes 2RA Q-learning a valuable addition to the RL tool belt.
The conducted numerical experiments for various settings corroborate our theoretical findings and highlight that 2RA Q-learning generally performs well and mostly outperforms other Q-learning variants.

\printbibliography


\appendix
\section{Linearization results} \label{app:linearization}

The proof of Theorem~\ref{thm:AMSE} uses a key result, stating that the asymptotic mean-squared error of the proposed 2RA Q-learning method \eqref{eq:q-learning:2RA} can be alternatively characterized by the linearized recursion~\eqref{eq:asymptotic:DRO:qlearning}. 
This result is a modification of the analysis for Double Q-learning derived in \citep[Appendix~A]{ref:Wentao-20}. 
To derive a formal statement, we introduce the key tool to analyze the linearization \eqref{eq:asymptotic:DRO:qlearning}, which is the ODE counterpart of the 2RA Q-learning scheme~\eqref{eq:q-learning:2RA} given as
\begin{align}\label{eq:nonlin:ODE}
    \dot \theta^{(i)}(t) = \lim_{n\to\infty} g \mathbb{E}[\phi(X_n)\big( r(X_n)-\phi(X_n)^\top \theta^{(i)}(t) \big) 
    + \gamma  \phi(X_n) (\max_{a'} \phi(S_{n+1},a') \widehat\theta_{N}(t) )],
\end{align}
where $\widehat\theta_{N}(t) = \frac{1}{N}\sum_{i=1}^N \theta^{(i)}(t)$ and where we have used $\lim_{n\to\infty} \rho_n =0$.
With the help of the ODE~\eqref{eq:nonlin:ODE} we can justify why working with the linearized system \eqref{eq:asymptotic:DRO:qlearning} allows us to quantify the AMSE for the 2RA Q-learning~\eqref{eq:q-learning:2RA}. 
This justification requires the following Assumption. 
We use the vectorized notation ${{\theta}} = (\theta^{{(1)}^\top},\dots,\theta^{{(N)}^\top})^\top$ and ${\mb{\theta^\star}=(\theta^{\star\top},\dots,\theta^{\star\top})^\top}$ where $\theta^\star\in\mathbb{R}^d$ is the optimal solution to the underlying MDP. 

\begin{assumption}[Linearization]\label{ass:linearization}
    We stipulate that for any $N\in\mathbb{N}$
    \begin{enumerate}[(i)]
        \item \label{ass:i:linearization} The process $\theta^{(i)}(t)$ described by the ODE~\eqref{eq:nonlin:ODE} has a globally asymptotically stable equilibrium $\bar\theta^{(i)}$ for any $i=1,\dots, N$.
        \item \label{ass:ii:linearization}The optimal policy $\pi^\star$ of the underlying MDP is unique.
        \item \label{ass:iii:linearization}The sequence of random variables $\{n \|\theta_n-
    {\mb{\theta^\star}}\|_2^2, n\geq 1 \}$ is uniformly integrable.
    \end{enumerate}
\end{assumption}
Sufficient conditions for Assumption~\ref{ass:i:linearization}, when using linear function approximation and in the setting $N=1$, are studied in \citet{ref:Melo-08,ref:Lee-20}. 
Assumption~\ref{ass:ii:linearization} is standard in many theoretical treatments of Q-learning and  Assumption~\ref{ass:iii:linearization} for $N=1$ has been established, see \citet[Theorem 5.5.2]{durrett_book}, \citet{ref:Devraj-17}. 
\begin{lemma}[Linearization] \label{lem:linearizaton}
    Let $\{\theta_n\}_{n\in\mathbb{N}}$ be a sequence 
    generated by the 2RA Q-learning~\eqref{eq:q-learning:2RA} and let $\{\bar\theta_n\}_{n\in\mathbb{N}}$ be a sequence generated by its linearized counterpart \eqref{eq:asymptotic:DRO:qlearning}. 
    Under Assumption~\ref{ass:linearization}, we have
    \begin{equation*}
        \lim_{n\to\infty} n \mathbb{E}[\|\theta_n-{\mb{\theta^\star}}\|_2^2]
        =\lim_{n\to\infty} n \mathbb{E}[\|\bar\theta_n-{\mb{\theta^\star}}\|_2^2].
    \end{equation*}
\end{lemma}

\begin{proof}
    In a first step, we show that the ODE~\eqref{eq:nonlin:ODE} for any $i=1,\dots,N$ has a unique globally asymptotically stable equilibrium given as $\bar{\theta}^{(i)} = \theta^\star$, where $\theta^\star$ is the the limit of Watkins' Q-learning \citep[Equation~(22)]{ref:Wentao-20}. 
    Note that by Assumption~\ref{ass:linearization}\ref{ass:i:linearization}, the ODE~\eqref{eq:nonlin:ODE} has a unique globally asymptotically stable equilibrium that we denote as $\bar\theta^{(i)}$ for all $i=1,\dots,N$. 
    By symmetry, any perturbation of it is a globally asymptotically equilibrium too. 
    Hence, by uniqueness we must have $\bar\theta^{(i)}=\bar\theta^{(j)}$ for all $i,j\in\{1,\dots,N\}$. 
    Since all equilibrium points are identical, we recover the equilibrium point of the ODE describing Watkins' Q-learning, i.e., $\bar\theta^{(i)} = \theta^\star$ for all $i=1,\dots,N$.
    
    Next, in order to apply \citep[Theorem~3]{ref:Wentao-20} with respect to the ODE~\eqref{eq:nonlin:ODE} we define for any $i=1,\dots,N$
    \begin{align*}
        w^{(i)}({ {\theta(t)}}) = \lim_{n\to\infty} g \mathbb{E}[\phi(X_n)\big( r(X_n)-\phi(X_n)^\top \theta^{(i)}(t) \big) 
        + \gamma  \phi(X_n) (\max_{a'} \phi(S_{n+1},a') \widehat\theta_{N}(t) )].
    \end{align*}
    With the vector notation $w( {\theta}) = (w^{(1)}({ {\theta}})^\top,\dots,w^{(N)}({ {\theta}})^\top)^\top$ and by plugging in the globally asymptotically stable equilibrium from above we obtain 
    \begin{align}\label{eq:linearized:ODE}
        w({\mb{\theta^\star}}) = g\bar{\mathsf{b}} + g (\bar{\mathsf{A}}_2-\bar{\mathsf{A}}_1) {\mb{\theta^\star}},
    \end{align}
    where $\bar{\mathsf{b}}$, $\bar{\mathsf{A}}_1$ and $\bar{\mathsf{A}}_2$ have been introduced in \eqref{eq:bar:symbol:A1:A2:b}. 
    Note that \eqref{eq:linearized:ODE} corresponds to the ODE of the linearized 2RA Q-learning \eqref{eq:asymptotic:DRO:qlearning} at the point ${\mb{{ {\theta^\star}}}}$. 
    We aim to show that $\nabla_{ {\theta}}w({\mb{\theta^\star}}) = g(\bar{\mathsf{A}}_2-\bar{\mathsf{A}}_1)$. 
    This result follows from observing that there exists $\varepsilon>0$ such that for any ${ {\theta}}$ such that $\|{ {\theta}}-{\mb{\theta^\star}}\|_\infty \leq  \varepsilon$ it holds 
    $w^{(i)}({ {\theta}}) = g\bar{\mathsf{b}} + g (\bar{\mathsf{A}}_2-\bar{\mathsf{A}}_1)\theta$. 
    To see why this is the case, by following \citet{ref:Wentao-20}, note that for the optimal policy $\pi^\star$ corresponding to $\theta^\star$ we can define
    \begin{equation*}
        \omega = \min_{(s,a)\in\mathcal{X}: a\neq\pi^\star(s)} \{ \phi(s,\pi^\star(s))^\top \theta^\star - \pi(s,a)^\top \theta^\star \} >0,
    \end{equation*}
    where the strict positivity follows from the uniqueness of $\pi^\star$. 
    Choose $\varepsilon = \frac{\omega}{3\|\Phi\|_1}$ and consider any $\theta^{(i)}\in\mathbb{R}^d$ such that $\|\theta^{(i)}-\theta^\star\|_\infty\leq \varepsilon$. 
    Then, for any $s\in\mathcal{S}$ and $a\in\mathcal{A}$ with $a\neq \pi^\star(s)$, it holds
    \begin{equation*}
        \phi(s,\pi^\star(s))^\top \theta^{(i)} - \phi(s,a)^\top \theta^{(i)} \geq 
        \phi(s,\pi^\star(s))^\top \theta^\star - \phi(s,a)^\top \theta^\star - 2\|\Phi^\top(\theta^{(i)}-\theta^\star)\|_\infty
        \geq \omega - \frac{2\omega}{3} >0,
    \end{equation*}
    which implies $\pi^\star = \pi_\theta$, i.e., for any $\theta$ such that $\|\theta-{\mb{ \theta^\star}}\|_\infty \leq \varepsilon$, the corresponding greedy policy $\pi_{\mb{ \theta^\star}}$ is optimal.
    Since $\|{{\theta^{(i)}}}-{{\theta^\star}}\|_\infty\leq \|{ {\theta}}-{ \mb{\theta^\star}}\|_\infty$, it indeed holds for any ${ {\theta}}$ such that $\|{ {\theta}}-{\mb {\theta^\star}}\|_\infty \leq \varepsilon$ that $w^{(i)}({ {\theta}}) = g\bar{\mathsf{b}} + g (\bar{\mathsf{A}}_2-\bar{\mathsf{A}}_1)\theta$. 
    So we have shown 
    \begin{equation}\label{eq:gradient:w}
        \nabla_{ {\theta}}w({ \mb{\theta^\star}}) = g(\bar{\mathsf{A}}_2-\bar{\mathsf{A}}_1).
    \end{equation}
    In the final step, we define
    \begin{align*}
        W^{(i)}(Z_n) =  \mathbb{E}[\phi(X_n)\big( r(X_n)-\phi(X_n)^\top \theta^\star \big) 
        + \gamma  \phi(X_n) (\max_{a'} \phi(S_{n+1},a') \theta^\star )],
    \end{align*}
    with the notation from \eqref{eq:symbol:A1:A2:b}
    and denote its vectorized version as ${\textbf{W}}(Z_n) = (W^{(1)}(Z_n)^\top,\dots,W^{(N)}(Z_n)^\top)^\top$. We then note that
    \begin{equation*}
        {\textbf{W}}(Z_n)   ={\mathsf{b}}(Z_n) +  ({\mathsf{A}}_2(Z_n)-{\mathsf{A}}_1(Z_n)){ \mb{\theta^\star}}.
    \end{equation*}
    Then, by definition of the asymptotic covariance
    \begin{equation}\label{eq:covariance:C}
        \begin{aligned}
            C_\theta({ \mb{\theta^\star}}) &= \sum_{n=-\infty}^\infty \mathbb{E}[(g{\textbf{W}}(Z_n) - w({\mb{\theta^\star}})(g{\textbf{W}}(Z_1) - w({\mb{\theta^\star}})^\top] \\
            &=g^2\sum_{n=-\infty}^\infty \mathbb{E}[{\textbf{W}}(Z_n){\textbf{W}}(Z_1)^\top] \\
            &=b^2 \mathsf{\Sigma_b}.
        \end{aligned}
    \end{equation}
    From the two results \eqref{eq:gradient:w} and \eqref{eq:covariance:C} by invoking \citep[Theorem~3]{ref:Wentao-20} we obtain
    \begin{equation}\label{eq:normality:ODE}
        \sqrt{n} ({ {\theta_n}}-{\mb{\theta^\star}}) \overset{d}{\to} \mathcal{N}(0,\Sigma),
    \end{equation}
    where $\Sigma$ is the unique solution to the Lyapunov equation~\eqref{eq:lyapounov:robust:ql}.
    By applying the Continuous Mapping Theorem to \eqref{eq:normality:ODE} we get
    \begin{equation}\label{eq:MSE:distr}
        n \|{ {\theta_n}}-{\mb{\theta^\star}}\|_2^2 \overset{d}{\to} \|X\|_2^2, \quad X\sim\mathcal{N}(0,\Sigma).
    \end{equation}
    Finally, combining \eqref{eq:MSE:distr} with Assumption~\ref{ass:linearization}\ref{ass:iii:linearization} according to \citep[Theorem 5.5.2]{durrett_book} ensures that
    \begin{equation*}
        \lim_{n\to\infty} n \mathbb{E}[\|{ {\theta_n}}-{\mb{\theta^\star}}\|_2^2] =\mathbb{E}[\|X\|_2^2] = \tr{\Sigma}.
    \end{equation*}
    When considering the linearization \eqref{eq:q-learning:2RA} instead of the 2RA Q-learning, by following the same lines without the need to linearize, we obtain
    \begin{equation*}
        \lim_{n\to\infty} n \mathbb{E}[\|{ {\bar\theta_n}}-{\mb{\theta^\star}}\|_2^2] =\mathbb{E}[\|X\|_2^2] = \tr{\Sigma},
    \end{equation*}
    where again $P$ is the unique solution to  the Lyapunov equation~\eqref{eq:lyapounov:robust:ql}. This completes the proof.
\end{proof}

\newpage

\section{Additional Numerical Results}\label{app:additional:experiments}
\subsection{Additional Plots for Random Environment Experiment}\label{app:plots:random:environment}

We provide some additional plots of the random experiment described in Section~\ref{ssec:random:environment}.

\begin{figure*}[h]
    \begin{subfigure}[c]{0.245\textwidth}
        \includegraphics[width=\textwidth]{figures/random_1.pdf}
        \subcaption{\label{fig:appendix_random_1}}
    \end{subfigure}
    \begin{subfigure}[c]{0.245\textwidth}
        \includegraphics[width=\textwidth]{figures/random_2.pdf}
        \subcaption{\label{fig:appendix_random_2}}
    \end{subfigure}
    \begin{subfigure}[c]{0.245\textwidth}
        \includegraphics[width=\textwidth]{figures/random_3.pdf}
        \subcaption{\label{fig:appendix_random_3}}
    \end{subfigure}
    \begin{subfigure}[c]{0.245\textwidth}
        \includegraphics[width=\textwidth]{figures/random_4.pdf}
        \subcaption{\label{fig:appendix_random_4}}
    \end{subfigure}
    \vskip\baselineskip
    \begin{subfigure}[c]{0.245\textwidth}
        \includegraphics[width=\textwidth]{figures/random_5.pdf}
        \subcaption{\label{fig:appendix_random_5}}
    \end{subfigure}
    \begin{subfigure}[c]{0.245\textwidth}
        \includegraphics[width=\textwidth]{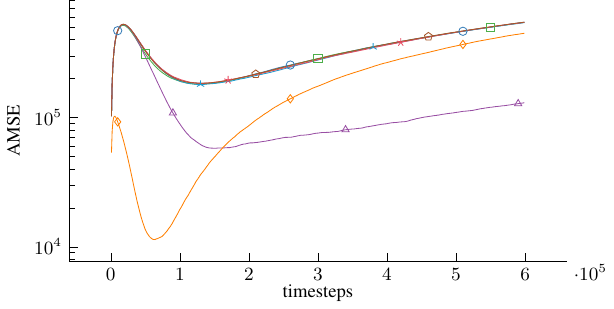}
        \subcaption{\label{fig:appendix_random_6}}
    \end{subfigure}
    \begin{subfigure}[c]{0.245\textwidth}
        \includegraphics[width=\textwidth]{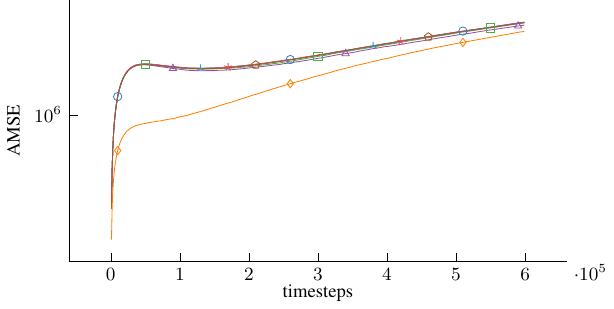}
        \subcaption{\label{fig:appendix_random_7}}
    \end{subfigure}    
    \begin{subfigure}[c]{0.245\textwidth}
        \includegraphics[width=\textwidth]{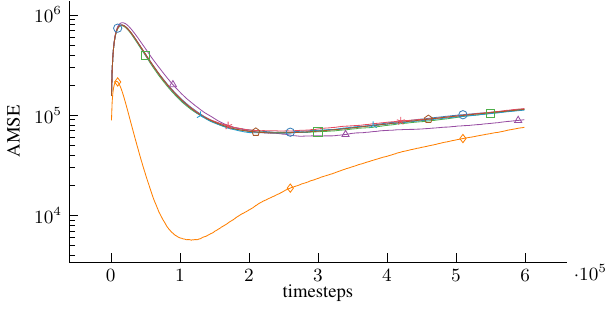}
        \subcaption{\label{fig:appendix_random_8}}
    \end{subfigure}
    \vskip\baselineskip
    \begin{subfigure}[c]{0.245\textwidth}
        \includegraphics[width=\textwidth]{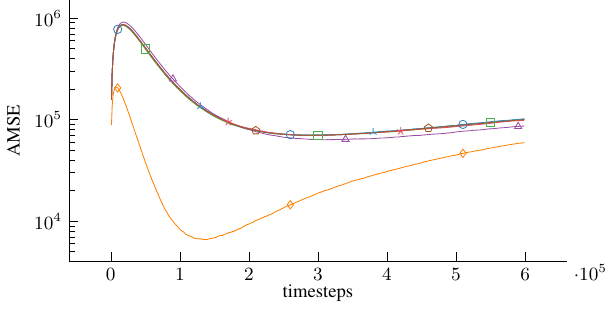}
        \subcaption{\label{fig:appendix_random_9}}
    \end{subfigure}
    \begin{subfigure}[c]{0.245\textwidth}
        \includegraphics[width=\textwidth]{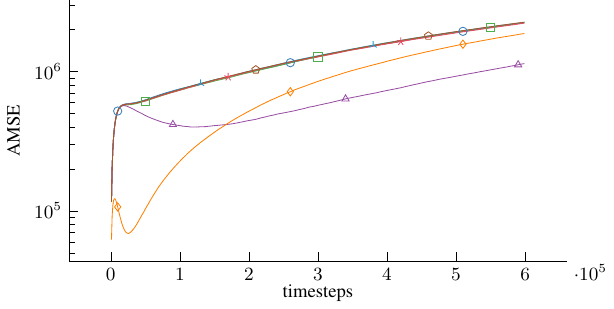}
        \subcaption{\label{fig:appendix_random_10}}
    \end{subfigure}
    \begin{subfigure}[c]{0.245\textwidth}
        \includegraphics[width=\textwidth]{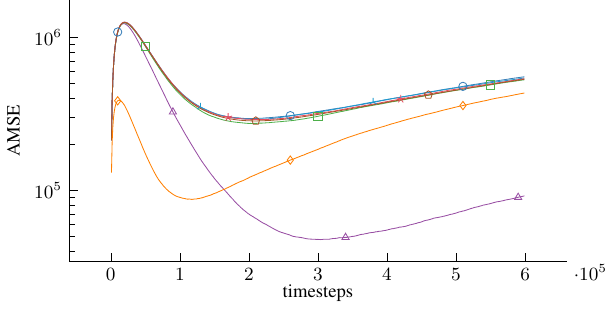}
        \subcaption{\label{fig:appendix_random_11}}
    \end{subfigure}    
    \begin{subfigure}[c]{0.245\textwidth}
        \includegraphics[width=\textwidth]{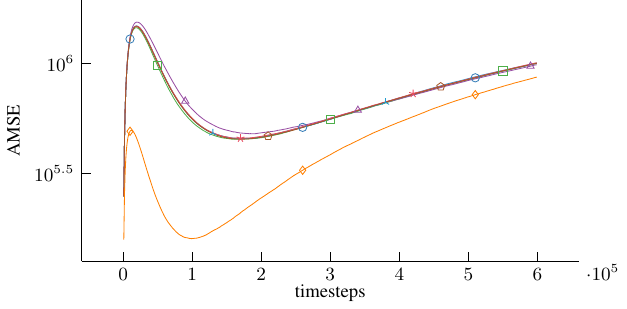}
        \subcaption{\label{fig:appendix_random_12}}
    \end{subfigure}
    \vskip\baselineskip
    \begin{subfigure}[c]{0.245\textwidth}
        \includegraphics[width=\textwidth]{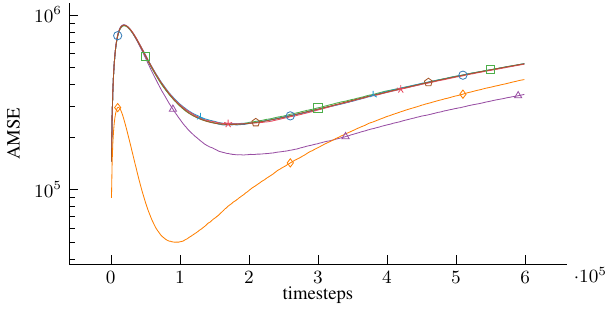}
        \subcaption{\label{fig:appendix_random_13}}
    \end{subfigure}
    \begin{subfigure}[c]{0.245\textwidth}
        \includegraphics[width=\textwidth]{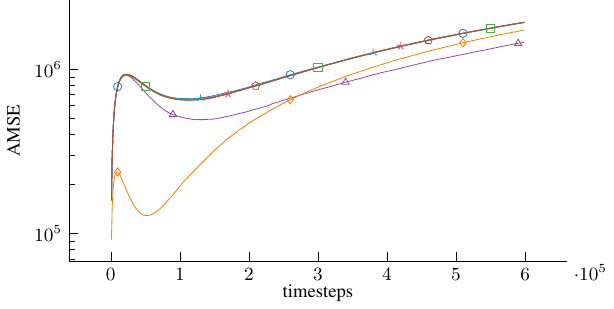}
        \subcaption{\label{fig:appendix_random_14}}
    \end{subfigure}
    \begin{subfigure}[c]{0.245\textwidth}
        \includegraphics[width=\textwidth]{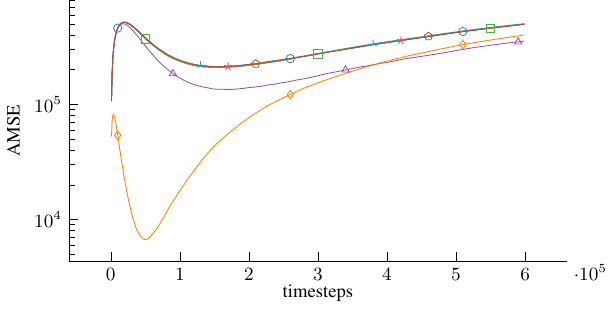}
        \subcaption{\label{fig:appendix_random_15}}
    \end{subfigure}    
    \begin{subfigure}[c]{0.245\textwidth}
        \includegraphics[width=\textwidth]{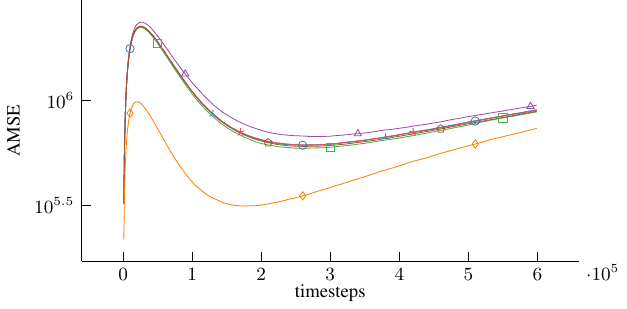}
        \subcaption{\label{fig:appendix_random_16}}
    \end{subfigure}
    \vskip\baselineskip
    \begin{subfigure}[c]{0.245\textwidth}
        \includegraphics[width=\textwidth]{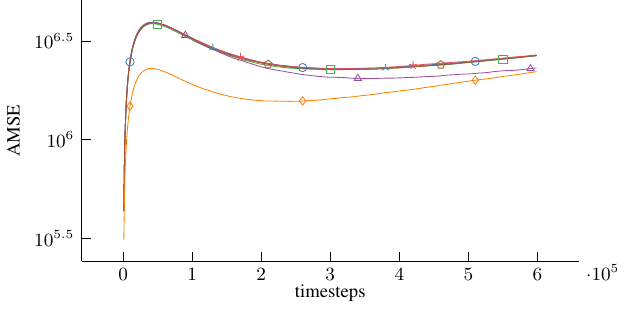}
        \subcaption{\label{fig:appendix_random_17}}
    \end{subfigure}
    \begin{subfigure}[c]{0.245\textwidth}
        \includegraphics[width=\textwidth]{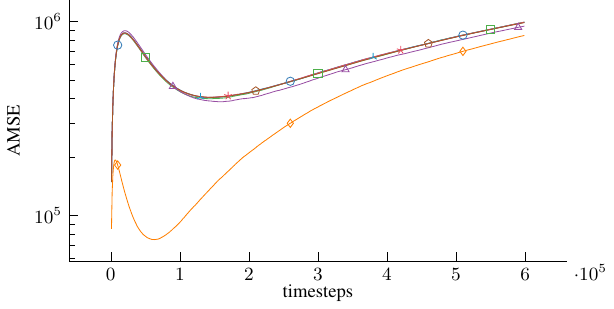}
        \subcaption{\label{fig:appendix_random_18}}
    \end{subfigure}
    \begin{subfigure}[c]{0.245\textwidth}
        \includegraphics[width=\textwidth]{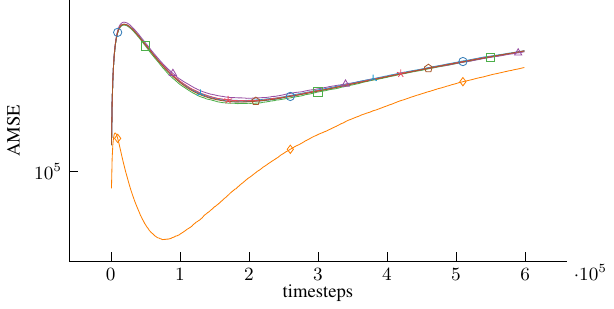}
        \subcaption{\label{fig:appendix_random_19}}
    \end{subfigure}    
    \begin{subfigure}[c]{0.245\textwidth}
        \includegraphics[width=\textwidth]{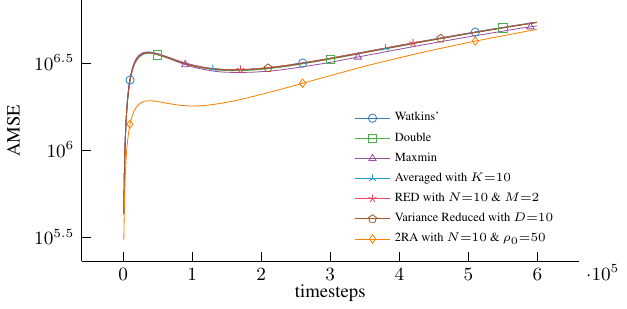}
        \subcaption{\label{fig:appendix_random_20}}
    \end{subfigure}
    \caption{Random Environment. All methods use an initial learning rate of $\alpha_0=0.01$, $w_{\alpha}=10^{5}$, $\gamma = 0.9$, and all $\theta^{(i)}$ initialized as zero. Maxmin as well as 2RA Q-learning have $N=10$ and 2RA agents additionally use $\rho_{0}=50$ and $w_{\rho}=10^{4}$. The plots show the first 20 randomly drawn environments.}
    \label{fig:appednix_random_env_figures}
\end{figure*}

\subsection{Neural Network Function Approximation} \label{app:ssec:neural:networks}

\begin{figure*}[t]
    \begin{subfigure}[b]{0.49\textwidth}
        \includegraphics[width=\textwidth]{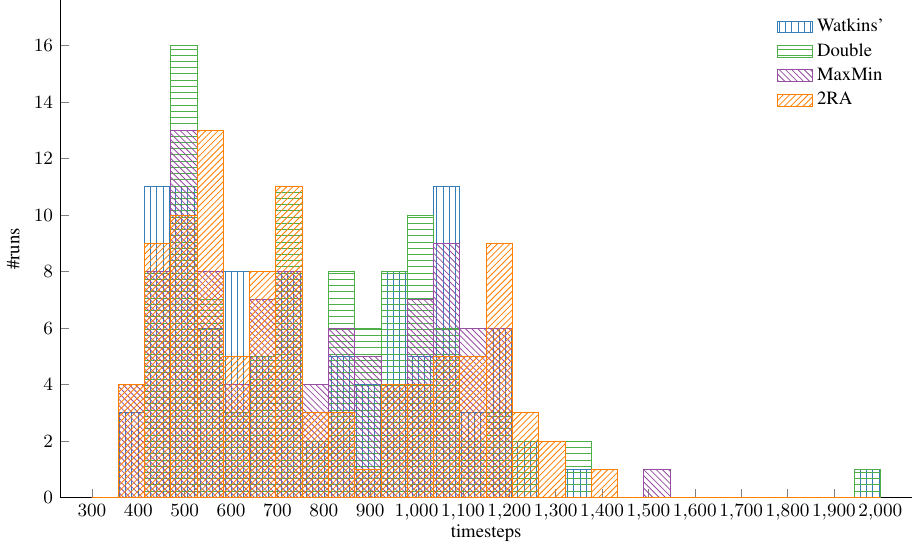}
        \subcaption{Distribution of hit times\label{fig:lunarlander_hist}}
    \end{subfigure}
    \hfill
    \begin{subfigure}[b]{0.49\textwidth}
        \centering
        \resizebox{0.75\textwidth}{!}{\begin{tabular}{ |l|c| } 
     \hline
     Algorithm & Mean hit time \\ 
     \hline
     Watkins' Q-Learning & $785 \pm 281.38$ \\ 
     Double Q-Learning & $791 \pm 268.64$ \\ 
     Maxmin Q-Learning & $770 \pm 252.19$ \\ 
     2RA Q-Learning & $767 \pm 276.24$ \\ 
     \hline
\end{tabular}
}\\
        \vspace*{1.5cm}
        \subcaption{Mean and std of hit times\label{fig:lunarlander_table}}
    \end{subfigure}
    \caption{LunarLander, 100 experiments. All methods use a learning rate of $\alpha=0.0002$ and a decay factor of $\gamma=0.99$. Maxmin, as well as 2RA Q-learning, have $N=5$. 2RA further uses $\rho_{0}=25$ and $w_{\rho}=10^{4}$. All algorithms are evaluated every $50$ episodes and recorded if the average evaluation reward reaches or exceeds 200. (a) Shows the distributions of each algorithm's hit times and (b) lists the respective mean hit times and corresponding standard deviations.}
    \label{fig:lunarlander_figures}
\end{figure*}

As an additional experiment, to test 2RA Q-learning when used with neural network Q-function approximation implemented in Tensorflow \citep{tensorflow2015-whitepaper}, the LunarLander environment \citep{gym2016} is chosen.
A lander receives a large positive reward for landing in a designated area, a large negative reward for crashing, and a small negative reward for firing a thruster.
Similar to the CartPole experiment, the latest updated policy with $\epsilon$-greedy exploration is used to generate the next timestep on which the model is updated since not all states may be reached by a random uniform policy. 
The environment is considered to be solved if the average reward over $100$ episodes, during evaluation, reaches or exceeds $200$.
Analogue to the CartPole experiment, different algorithms are compared based on how many training episodes are required to solve the LunarLander environment where fewer average timesteps, until the environment is solved, result in a higher performance ranking for the corresponding model.
Instead of the $\theta^{(i)}$ a small neural network with two hidden ReLU \citep{ref:He2015} layers and $N$ sets of weights are used.
Averaging operations that were performed on the $\theta^{i}$ in the linear function approximation scenarios are now performed on sets of weights for the neural network.
All models are trained with a Huber loss \citep{ref:huber1964} and the Adam optimizer \citep{ref:KingmaBa2015} using a learning rate of $\alpha_{0} = 0.0002$.
The different learning methods are implemented as plain and close to the theory as possible.
Updates are therefore applied after every timestep and on that single timestep.

Comparing the hit times shows only little difference between all learning methods with Maxmin and 2RA Q-learning leading the field by a small margin.
Future work could aim to implement and test the method with more contemporary training pipelines, such as the incorporation of experience replay etc., to analyse whether such optimizations amplify the differences in performance or just apply a uniform shift to all learning methods.

\end{document}